\documentclass{article}
\usepackage{amsmath,amssymb,amsthm, enumitem, comment}
\newtheorem{conj}{Conjecture}
\newtheorem{thm}[conj]{Theorem}

\newtheorem*{claim*}{Claim}
\newtheorem{lemma}[conj]{Lemma}
\newtheorem{prop}[conj]{Proposition}

\newtheorem{cor}[conj]{Corollary}

\theoremstyle{definition}
\newtheorem{definition}[conj]{Definition}
\newtheorem{cons}[conj]{Construction}
\newtheorem{ques}[conj]{Question}

\usepackage{caption}
\usepackage{subcaption}

\usepackage{amssymb, tikz}
\usetikzlibrary{shapes}
\newcommand{\Cc}{\mathcal C}
\newcommand{\Tt}{\mathcal T}
\newcommand{\Ss}{\mathcal S}
\newcommand{\Pp}{\mathcal P}

\tikzstyle{vertex}=[shape=circle, minimum size=2mm, fill, inner sep=0]
\usepackage[notref,notcite]{showkeys}
\usepackage{fullpage}
\begin{document}

\title{A refinement of theorems on vertex-disjoint chorded cycles}
\author{Theodore Molla\thanks{Department of Mathematics, University of Illinois, Urbana, IL 61801, USA.
This author's research is supported in part by the NSF grant DMS-1500121.
E-mail address: \texttt{molla@illinois.edu}},~
Michael Santana\thanks{Department of Mathematics, University of Illinois, Urbana, IL 61801, USA.
This author's research is supported in part by the NSF grant DMS-1266016 ``AGEP-GRS''.
E-mail address: \texttt{santana@illinois.edu}},~
Elyse Yeager\thanks{Mathematics Department, University of British Columbia, Canada.
E-mail address: \texttt{elyse@math.ubc.ca}}
}

\maketitle
\begin{abstract}{In 1963, Corr\'adi and Hajnal settled a conjecture of Erd\H{o}s by proving that, for all $k \geq 1$, any graph $G$ with $|G| \geq 3k$ and minimum degree at least $2k$ contains $k$ vertex-disjoint cycles. In 2008, Finkel proved that for all $k \geq 1$, any graph $G$ with $|G| \geq 4k$ and minimum degree at least $3k$ contains $k$ vertex-disjoint chorded cycles.
Finkel's result was strengthened by Chiba, Fujita, Gao, and Li in 2010, who showed, among other results,
that for all $k \geq 1$, any graph $G$ with $|G| \geq 4k$ and minimum Ore-degree at least $6k-1$ contains $k$ vertex-disjoint cycles.
We refine this result, characterizing the graphs $G$ with $|G| \geq 4k$ and minimum Ore-degree at least $6k-2$ that do not have $k$ disjoint chorded cycles. %This refinement mirrors a result by Kierstead, Kostochka, and Yeager in 2015 characterizing graphs $G$ with $|G| \geq 3k$ and minimum degree at least $2k-1$ that do not have $k$ vertex-disjoint cycles.
}\end{abstract}

{\small{Mathematics Subject Classification: 05C35, 05C38, 05C75.}}{\small \par}

{\small{Keywords: Disjoint cycles, chorded cycles,  minimum degree, Ore-degree.}}{\small \par}

\section{Introduction}

% TM double quotes here - right?
All graphs in this paper are simple, unless otherwise noted.  Additionally, when referring to cycles in a graph, ``disjoint'' is always taken to mean ``vertex-disjoint.''  For a graph $G$, we use $V(G)$ and $E(G)$ to denote the vertices and edges, respectively, and for a vertex $v$, we will use $v \in G$ to denote $v \in V(G)$.   For a vertex $v \in G$, and for a subgraph $H$ of $G$ (where possibly $H = G$), the neighborhood of $v$ in $H$ is denoted by $N_H(v)$.  The number of neighbors of $v$ in $H$ (i.e., $|N_H(v)|$) will be written by $d_H(v)$.   Furthermore, we write $|G|$ for the order of a graph $G$, $\overline{G}$ for its complement,  $\delta(G)$ for its minimum degree, and $\alpha(G)$ for its independence number. 

The minimum Ore-degree of a non-complete graph $G$ is written $\sigma_2(G)$, and defined as
\begin{equation*}
  \sigma_2(G):=\min\{d_G(x)+d_G(y):xy \in E(\overline G)\}; 
\end{equation*}
that is, $\sigma_2(G)$ is the minimum degree-sum of nonadjacent vertices.
$K_n$ is the complete graph on $n$ vertices, and $K_{s_1,\ldots,s_t}$ is the complete $t$-partite graph with parts of size $s_1, \ldots,s_t$. For graphs $G$ and $H$, $G+H$ is the disjoint union of $G$ and $H$, and $G \vee H$ is the join of $G$ and $H$.

In 1963, Corr\'adi and Hajnal verified a conjecture of Erd\H{o}s, proving the following.
\begin{thm}[Corr\'adi-Hajnal, \cite{CH}]\label{CHT}
Every graph $G$ on $|G| \geq 3k$ vertices with $\delta(G) \geq 2k$ contains $k$ disjoint cycles.
\end{thm}

This result of Corr\'{a}di and Hajnal has been generalized in various ways.  One such generalization is a strengthening by Enomoto and Wang, who independently proved the following.

\begin{thm}[Enomoto \cite{E} , Wang \cite{W}]\label{EW}
Every graph $G$ on $|G| \ge 3k$ vertices with $\sigma_2(G) \ge 4k - 1$ contains $k$ disjoint cycles.
\end{thm}

Both Theorems \ref{CHT} and \ref{EW} are sharp, leading to the following natural question of Dirac.

\begin{ques}[Dirac, \cite{Di}]\label{DiracQ}
Which $(2k-1)$-connected multigraphs do not contain $k$ disjoint cycles?
\end{ques}

Question~\ref{DiracQ} was answered in the case of simple graphs in~\cite{KKY}, and then in multigraphs in~\cite{KKY2}. Indeed, \cite{KKY} together with \cite{KKMY} answer a more general question for simple graphs, describing graphs with minimum Ore-degree at least $4k - 3$ with no $k$ disjoint cycles. To avoid going into too many technical details, we only provide part of this description below.

\begin{thm}[\cite{KKY}, \cite{KKMY}]\label{KKYT}
Given an integer $k \geq 4$, let $G$ be a graph on $|G| \ge 3k$  vertices with $\sigma_2(G) \ge 4k - 3$. Then $G$ contains $k$ disjoint cycles if and only if none of the following hold:
\begin{enumerate}
\item $\alpha(G) \ge |G| - 2k+1$.
\item $G=(K_c + K_{2k-c}) \vee \overline{K_k}$ for some odd $c$
\item $G=(K_1 + K_{2k}) \vee \overline{K_{k-1}}$
\item $|G|=3k$ and $\overline{G}$ is not $k$-colorable
\end{enumerate}
\end{thm}

%MS- These papers by themselves don't provide this characterization.  We need papers by Kostochka, Rabern, Stiebitz, Postle, etc. 
%Papers~\cite{KKY} and \cite{KKMY} together characterize an even broader class of graphs: those graphs $G$ with $|G| \geq 3k$ and
%$\sigma_2(G) \geq 4k-3$ that do not contain $k$ disjoint cycles.

In 2008 Finkel proved the following chorded-cycle analogue to Theorem~\ref{CHT}.
\begin{thm}[Finkel, \cite{F}]\label{FinkelT}
Every graph $G$ on $|G| \geq 4k$ vertices with $\delta(G) \geq 3k$ contains $k$ disjoint chorded cycles.
\end{thm}

A stronger vertion of Theorem~\ref{FinkelT} was conjectured by Bialostocki, Finkel, and
Gy\'arf\'as in \cite{BFG}, and proved by Chiba, Fujita, Gao, and Li in \cite{CFGL}.

\begin{thm}[Chiba-Fujita-Gao-Li, \cite{CFGL}]\label{CFGLT}
Let $r$ and $k$ be integers with $r+k \geq 1$. Every graph $G$ on $|G| \geq 3r+4k$ vertices with $\sigma_2(G) \geq 4r+6k-1$ contains a collection of $r+k$ disjoint cycles such that  $k$ of these cycles are chorded.
\end{thm}

 In particular, the following corollary holds.

\begin{cor}[Chibia-Fujita-Gao-Li, \cite{CFGL}]\label{CFGLC}
Every graph $G$ on $|G| \ge 4k$ vertices with $\sigma_2(G) \ge 6k - 1$ contains a collection of $k$ disjoint chorded cycles.
\end{cor}

All hypotheses in Theorem \ref{FinkelT} and Corollary \ref{CFGLC} are sharp. First, since any chorded cycle contains at least four vertices, if $|G|<4k$ then $G$ does not contain $k$ disjoint chorded cycles. Second, the conditions $\delta(G) \geq 3k$ and $\sigma_2(G) \ge 6k - 1$ are best possible, as demonstrated by the two graphs below.

 \begin{definition}\label{G1}
 For   $n \geq 6k-2$, define
 $G_1(n,k):=K_{3k-1,n-3k+1}$  (Figure~\ref{G1F}).
 For $k \geq 2$, define $G_2(k):=K_{3k-2,3k-2,1}$ (Figure~\ref{G2F}). 
   \end{definition}
For  $n \geq 6k-2$, $|G_1(n,k)|=n \geq 4k$ and $\sigma_2(G_1(n,k)) =6k-2$. Each chorded cycle in $G_1(n,k)$ uses at least three vertices from each part, so $G_1(n,k)$ does not contain $k$ disjoint chorded cycles.
For $k \geq 2$, $|G_2(k)|=6k-3 \geq 4k$ and $\sigma_2(G_2(k)) =6k-2$. Each chorded cycle in $G_2(k)$ uses three vertices from each of the big parts, or the dominating vertex and at least two vertices from a big part, so $G_2(k)$ does not contain $k$ chorded cycles.
 
 \begin{figure}[ht]
 \centering
 \begin{subfigure}[b]{.4\textwidth}
 \centering
 \begin{tikzpicture}
 \draw (0,-.25) node{$3k-1$};
 \draw (2,-.5) node{$n-3k+1$};
 \foreach \z in {1,2,3,4,5}{
 	\draw (0,-0.0225+.5*\z) node[vertex](a\z){};}
\foreach \x in {0,1,5,6}{
	\draw(2,.5*\x) node[vertex](b\x){};
	\foreach \y in {1,2,3,4,5}
		{\draw (a\y)--(b\x);}}
\foreach \x in {1,2,3}{
	\draw (2,1+.25*\x) node[shape=circle, minimum size=1mm, inner sep=0, fill, draw]{};}
 \end{tikzpicture}
 \caption{$G_1(n,k)$, shown for $k=2$}\label{G1F}
 \end{subfigure}
 \begin{subfigure}[b]{.4\textwidth}\centering
  \begin{tikzpicture}
   \draw (0,0) node{$3k-2$};
 \draw (2,0) node{$3k-2$};
\draw (1,3) node[vertex](c){};
 \foreach \x in {1,2,3,4}{
 	\draw (0,.5*\x) node[vertex](a\x){};
 	\draw (a\x)--(c);}
\foreach \x in {1,2,3,4}{
	\draw(2,.5*\x) node[vertex](b\x){};
	\draw (b\x)--(c);
	\foreach \y in {1,2,3,4}
		{\draw (a\y)--(b\x);}}
  \end{tikzpicture}
 \caption{$G_2(k)$, shown for $k=2$}\label{G2F}
 \end{subfigure}
 \caption{Graphs $G_1(n,k)$ and $G_2(k)$ from Definition~\ref{G1}.}
 \end{figure}

We can now ask a question similar to Question~\ref{DiracQ}: 
which graphs $G$ with $\sigma_2(G) \geq 6k-2$ do \emph{not} contain $k$ disjoint chorded cycles? Our main result is the following.

\begin{thm}\label{main} 
For  $k \geq 2$,
let $G$ be a graph with $n:=|G| \geq 4k$ and $\sigma_2(G) \geq 6k-2$.  $G$ does not contain $k$ disjoint chorded cycles if and only if $G \in \{G_1(n,k),G_2(k)\}$.
\end{thm}

The condition $k \geq 2$ in Theorem~\ref{main} is necessary, as subividing every edge of a graph results in a new graph with no chorded cycles.  Thus, for $k = 1$, we obtain the following characterization, which is analogous to the characterization of acyclic graphs as the graphs for which there exists at most one path between every pair of vertices.

\begin{prop}
A graph $G$ has no chorded cycle if and only if for all $uv \in E(G)$, $G - uv$ has at most one path between $u$ and $v$.
\end{prop}

Every graph $G$ with $\delta(G) \geq 3k-1$ also satisfies $\sigma_2(G) \geq 6k-2$. Therefore, Theorem~\ref{main} is a refinement of both Theorem~\ref{FinkelT} and Corollary \ref{CFGLC}. Two other immediate corollaries of Theorem~\ref{main} are listed here.

\begin{cor}\label{indep cor}
 For $k \geq 2$,
let $G$ be a graph with $|G| \geq 4k$, $\sigma_2(G) \geq 6k-2$, and $\alpha(G) \leq n-3k$. Then $G$ contains $k$ disjoint chorded cycles.
\end{cor}

Every graph $G$ with $\sigma_2(G) \geq 6k-2$ also satisfies $\alpha(G) \leq n-3k+1$. So, requiring $\alpha(G) \leq n-3k$ in Corollary~\ref{indep cor} is equivalent to requiring the seemingly weaker condition $\alpha(G) \neq n-3k+1$.

\begin{cor} 
For $k \geq 2$,
let $G$ be a graph with $4k \leq |G| \leq 6k-4$ and $\sigma_2(G) \geq 6k-2$. Then $G$  contains $k$ disjoint chorded cycles.
\end{cor}

\subsection{Outline}

We present our result as follows.  In Section \ref{setup}, we detail the setup of our proof and present several important lemmas that will be used throughout our paper.  In particular, we find and choose an `optimal' collection of $k - 1$ disjoint cycles, and use $R$ to denote the subgraph induced by the vertices outside our collection.  Then, in Section~\ref{R ne P}, we consider the case when $R$ does not have a spanning path, and, in Section~\ref{long path}, we consider the case when $R$ has a spanning path.  We conclude our paper in Section \ref{remarks} with some remarks on further extensions.

%%%%%%%%%%%%%%%%%%%%%%%%%%%%%%%%%%%%%%
%%%%%%%%%%%%%%%%%%%%%%%%%%%%%%%%%%%%%%
%%%%%%%%%%%%%%%%%%%%%%%%%%%%%%%%%%%%%%
\section{Setup and Preliminaries}\label{setup}
%%%%%%%%%%%%%%%%%%%%%%%%%%%%%%%%%%%%%%
%%%%%%%%%%%%%%%%%%%%%%%%%%%%%%%%%%%%%%
%%%%%%%%%%%%%%%%%%%%%%%%%%%%%%%%%%%%%%

\subsection{Notation}

Let $G$ be a graph, and let $A, B \subseteq V(G)$, not necessarily disjoint.  We define $\|A,B\| := \sum\limits_{a \in A} |N_G(a) \cap B|$.  When $A = \{a\}$ or $A$ is the vertex set of some subgraph $\mathcal A$, we will often replace $A$ in the above notation with $a$ or $\mathcal A$, respectively.  Additionally, if $\mathcal{L}$ is a collection of graphs, then $\|A, \mathcal{L}\| = \|A, \bigcup\limits_{L \in \mathcal{L}} V(L)\|$.  If $A$ is the vertex set of some subgraph $\mathcal A$, we will write $G[\mathcal A]$ for $G[A]$, the subgraph of $G$ induced by the vertices of $\mathcal A$.  Furthermore, if $\mathcal B$ is a subgraph of $G$ with vertex set $B$, we will use $\mathcal A\setminus \mathcal B$ to denote $G[A\setminus B]$, and if $B = \{b_1,\dots, b_k\}$ and $k$ is small, we will also use $\mathcal A - b_1 - \dots - b_k$.  For a vertex $v$, we additionally write $\mathcal A + v$ for $G[A \cup \{v\}]$.

If $P = v_1 \dotsc v_m$ is a path, then
for $1 \le i \le j \le m$, $v_iPv_j$ is the path $v_i \dotsm v_j$.
%Given vertices $v_1,v_2$ on a path $P$ sharing a component of $G$, $v_1Gv_2$ represents any path in $G$ from $v_1$ to $v_2$. 
An $n$-cycle is a cycle with $n$ vertices. A singly chorded cycle is a cycle with precisely one chord, and  a doubly chorded cycle is a cycle with at least two chords.

\subsection{Setup}

We let $k \geq 2$ and consider a graph $G'$ on $n$ vertices such that $n \ge 4k$ and $\sigma_2(G')=6k-2$, where  $G'$ does not contain $k$ disjoint chorded cycles. 
We then let $G$ be a graph with vertex set $V(G')$ such that 
$E(G') \subseteq E(G)$ and $G$ is ``edge-maximal'' in the
sense that, for any $e \in E(\overline{G})$, $G+e$ does contain $k$ disjoint chorded cycles.
We then prove that $G$ is $G_1(n,k)$ or $G_2(k)$, which
implies that $G = G'$, because
any proper spanning subgraph of $G_1(n,k)$ or $G_2(k)$
has minimum Ore-degree less than $6k - 2$.
Since we have already observed that $G_1(n,k)$ and $G_2(k)$ 
do not contain $k$ disjoint chorded cycles, this will prove Theorem~\ref{main}.

Note that $G \not\cong K_n$, else $G$ contains $k$ disjoint chorded cycles.  So there exists $e \in E(\overline{G})$, and by our edge-maximality condition, $G$ contains $k-1$ disjoint chorded cycles. 
%Let $\Cc$ be a set of $k-1$ disjoint chorded cycles in $G$, and let $R:=G\setminus \Cc$.
% be the subgraph of $G$ induced by the vertices not in $\Cc$.
Over all possible collections of $k-1$ disjoint chorded cycles in $G$, 
let $\Cc$ be such a collection which satisfies the following conditions 
when $R := G \setminus \Cc$:
\begin{enumerate}
\item[(O1)] the number of vertices in $\Cc$ is minimum,
\item[(O2)] subject to (O1), the total number of chords in the cycles of $\Cc$ is maximum, and
\item[(O3)] subject to (O1) and (O2), the length of the longest path in $R$ is maximum.
\end{enumerate}

We use the convention that $P$ is a longest path in $R$. 
% We use $p$ and $p'$ for the endpoints of $P$, and $q$ resp. $q'$ for the neighbor of $p$ resp. $p'$ along $P$. 
Since $G[P]$ may have several paths spanning $V(P)$ and the endpoints of such paths will behave in a similar manner, we let
\begin{equation*}
  \Pp:= \{v \in V(P) : v \text{ is an endpoint of a path spanning $V(P)$}\}.
\end{equation*}

%In Section 3 we use p and p', but in Section 4 we use p_1 and p_2, and use p,p' for arbitrary endpoints of other paths.  So I would rather start each section with what we'll use for the endpoints in that section rather than define it here.

\begin{comment}
\subsection{Outline}
In the rest of this section, we prove some basic results about the structure of $G$, and give some helpful lemmas about graphs with no chorded cycles.  In Section~\ref{R ne P}, we consider the case that $P$ does not span $R$, showing that in this case $G=G_1(n,k)$.
 In Section~\ref{long path}, we consider the case that $P$ spans $R$.
 %and $|P| \geq 6$, and show that it leads to a contradiction.
%  In Section~\ref{short path}, we consider the case that $P$ spans $R$ and $|P| \leq 5$. In this case, we show $G \in \{G_1(n,k),G_2(k)\}$. In the conclusion, we discuss the sharpness of Theorem~\ref{main}.
\end{comment}

\subsection{Preliminary Results}
% TM - a reorganized this slightly and changed one of the last observations slightly.  I think this is clearer.
%We begin with a number of observations about $G$ that follow directly from our setup. In the interest of readability, the observations in this paragraph will be used in the text without citation. If $p$ has a neighbor in $R\setminus P$, we can extend $P$, so $\|p,R\|=\|p,P\|$. If $\|p,P\| \geq 3$, then $G[P]$ contains a chorded cycle, so $\|p,R\| \leq 2$. Similarly, to avoid a chorded cycle in $R$, $\|q,P\| \leq 3$ and for any $r \in P$, $\|r,P\| \leq 4$. If $p$ has two neighbors in $P$, then we can find a path $P'$ of length $|P|$ in $R$ with $V(P)=V(P')$ and $p'$ as one endpoint.  Since $G$ does not contain $k$ disjoint chorded cycles, for any $C \in \Cc$, $G[R \cup C]$ does not contain two disjoint chorded cycles, and $R$ does not contain any chorded cycle.

We begin with a number of observations about $G$ that follow directly from our setup. In the interest of readability, the observations in this paragraph will be used in the text without citation. Since $G$ does not contain $k$ disjoint chorded cycles, $R$ does not contain any chorded cycle, and for any $C \in \Cc$, $G[R \cup C]$ does not contain two disjoint chorded cycles.  If $p$ is an endpoint of $P$ and has a neighbor in $R\setminus P$, we can extend $P$.  Thus, $\|p,R\|=\|p,P\|$. If $\|p,P\| \geq 3$, then $G[P]$ contains a chorded cycle, so $\|p,R\| \leq 2$. Similarly, to avoid a chorded cycle in $R$, $\|q,P\| \leq 3$ and for any $v \in P$, $\|v,P\| \leq 4$. If $p$ has two neighbors in $P$, then $G[P]$ contains two distinct spanning paths.

An immediate corollary of (O1) is that, for any chorded cycle $C \in \mathcal C$,
no vertex of $C$ is incident to two chords; otherwise, we could replace $C$ with a chorded cycle on fewer vertices. We will assume this fact in the proof of the following lemma.

\begin{lemma}\label{RCedges}
Let $v \in R$ and $C \in \Cc$.
\begin{enumerate}[label=\textup{(\arabic*)}]
\item\label{4edgesL} If $\|v,C\| \geq 4$, then $\|v,C\|=4=|C|$,  and $G[C] \cong K_4$. 
\item\label{3edgesL} If $\|v,C\| =3$, then $|C| \in \{4,5,6\}$. Moreover:
\begin{enumerate}[label=\textup{(\alph*)}, ref=\alph*]
\item\label{3edgesL4} if $|C|=4$, then $C$ has a chord incident to the non-neighbor of $v$ (see Figure~\ref{c4v3});
\item\label{3edgesL5} if $|C|=5$, then $C$ is singly chorded, and the endpoints of the chord are disjoint from the neighbors of $v$ (see Figure~\ref{c5v3}); and
\item\label{3edgesL6} if $|C|=6$, then $C$ has three chords, with $G[C]\cong K_{3,3}$ and $G[C+v]\cong K_{3,4}$ (see Figure~\ref{c6v3}).
\end{enumerate}
\end{enumerate}
\end{lemma}

\begin{figure}[ht]
\begin{subfigure}[b]{.3\textwidth}
\centering
\begin{tikzpicture}
\foreach \x in {1,2,3,4}{
\draw (90*\x:.75cm) node[vertex](c\x){};}
\draw (c4) node[label=right:$C$]{};
\draw (c1)--(c2)--(c3)--(c4)--(c1)--(c3);
\draw (0,-2) node[vertex](v){};
\draw (v) node[label=right:$v$]{};
\draw (c2)--(v)--(c3) (v)--(c4);
%\draw[dashed] (c2)--(c4);
\end{tikzpicture}
\caption{$|C|=4$, $\|v,C\|=3$}\label{c4v3}
\end{subfigure}
\begin{subfigure}[b]{.3\textwidth}
\centering
\begin{tikzpicture}
\foreach \x in {1,2,3,4,5}{
\draw (18+72*\x:.75cm) node[vertex](c\x){};}
\draw (c5) node[label=right:$C$]{};
\draw (c1)--(c2)--(c3)--(c4)--(c5)--(c1) (c2)--(c5);
\draw (0,-2) node[vertex](v){};
\draw (v) node[label=right:$v$]{};
\draw (c4)--(v)--(c3);
\draw (v) .. controls (-1,-1.5) and (-2,.75) .. (c1);
%\draw (v) to[out=150, in=180] (c1);
\end{tikzpicture}
\caption{$|C|=5$, $\|v,C\|=3$}\label{c5v3}
\end{subfigure}
\begin{subfigure}[b]{.3\textwidth}
\centering
\begin{tikzpicture}
\foreach \x in {1,2,3,4,5,6}{
\draw (30+60*\x:.75cm) node[vertex](c\x){};}
\draw (c6) node[label=right:$C$]{};
\draw (c1)--(c2)--(c3)--(c4)--(c5)--(c6)--(c3) (c2)--(c5) (c6)--(c1)--(c4);
\draw (0,-2) node[vertex](v){};
\draw (v) node[label=right:$v$]{};
\draw (c4)--(v);
\draw (v) to[out=135, in=225] (c2);
\draw (v) to[out=45, in=-45] (c6);
\end{tikzpicture}
\caption{$|C|=6$, $\|v,C\|=3$}\label{c6v3}
\end{subfigure}
\caption{Lemma~\ref{RCedges}\ref{3edgesL}}
\end{figure}

\begin{proof}
If there exist vertices $c_1,c_2 \in C$ that are adjacent along the cycle of $C$ such that $\|v,C-c_1-c_2\| \geq 3$, then $(C-c_1-c_2) + v$ contains a chorded cycle with strictly fewer vertices than $C$,  contradicting (O1). 
This proves that if $\|v, C\| = 3$, then $|C| \le 6$.  Similarly, if
$\|v, C\| \ge 4$, then $|C| = 4$ and $\|v, C\| = 4$.
%\ref{3edgesL} and the first part of \ref{4edgesL}. 
If $\|v,C\|=4$ and $|C|=4$, then $v$ together with a triangle in $C$ give a doubly chorded 4-cycle, so by (O2), $G[C] \cong K_4$.

Suppose $\|v,C\|=3$. If $|C|=4$, then let $c \in C$ be the non-neighbor of $v$ in $C$. If $c$ is not incident to a chord, then $(C-c)+v$ gives a doubly chorded 4-cycle, preferable to $C$ by (O2). This proves~(\ref{3edgesL4}).

So $|C|\in\{5,6\}$.  Since the vertices in $V(C)\setminus N_G(v)$ cannot be adjacent along the cycle $C$, $C-c+v$ contains a chorded cycle $C'$ of the same length as $C$, for any $c \in V(C)\setminus N_G(v)$,
%does not contain two vertices $c_1,c_2$ that are adjacent along the cycle: otherwise, $(C-c_1-c_2)+v$ contains a chorded cycle shorter than $C$, violating (O1).
 If $c$ is not incident to a chord, then $C'$ has strictly more chords than $C$, violating (O2). So every vertex in $V(C)\setminus N_G(v)$ is incident to a chord.

If $|C|=6$, then $v$ is adjacent to every other vertex along the cycle, and every $c \in V(C)\setminus N_G(v)$ is incident to a chord. Since no vertex in $C$ is incident to two chords, (O1) implies (\ref{3edgesL6}). If $|C|=5$, then (O1) implies that the only possible chord has the two non-neighbors of $v$ as its endpoints, which proves (\ref{3edgesL5}).
\end{proof}

\begin{lemma}\label{C4bounds}
  Let $Q$ be a path in $R$ such that $|Q| \ge 4$ and let $C \in \Cc$.  
  If $F \subseteq V(Q)$ such that $|F| = 4$, then
 $\|F, C\| \le 12$.
 Furthermore, if $G[C] \cong K_4$ and there exists
an endpoint $v$ of $Q$ such that $\|v, C\| \ge 3$,
then $\|Q, C\| \le 12$ with $\|Q, C\| = 12$ only if $\|v, C\| = 4$.
\end{lemma}
\begin{proof}
  Assume $\|F, C\| \ge 13$ for some $F \subseteq V(Q)$, $|F| = 4$, and  let $u_1, u_2, u_3, u_4$ be the vertices of $F$ in the order they appear on the path $Q$.
  By Lemma~\ref{RCedges}, $G[C] \cong K_4$, so
  there exists $c \in C$ such that $\|c, F\| \ge 4$.
   Since $\|\{u_1, u_4\}, C\| \ge 5$, there exists $i \in \{1, 4\}$ such
  that $\|u_i, C\| \ge 3$.
  So $Q - u_i + c$ and $C - c + u_i$ both contain chorded cycles, a contradiction.

  To prove the second statement, suppose $G[C]\cong K_4$ and let $v$ be an endpoint of $Q$ such that $\|v,C\| \ge 3$. Note that for every $c \in C$,
  $C - c + v$ and $Q - v + c$ both contain chorded cycles if 
  $\|c, Q - v\| \ge 3$.  Thus, $\|Q,C\| \le 12$, and furthermore,  if $\|Q, C\| = 12$, then $\|c, Q\| = 3$ and $\|c, v\| = 1$
  for every $c \in C$.
\end{proof}

\begin{lemma}\label{noC5}
  If $C \in \Cc$ and $\|v_1, C\|, \|v_2, C\| \ge 3$ for distinct $v_1, v_2 \in R$, then $|C| \in \{4, 6\}$.
\end{lemma}
\begin{proof}
  If $C \notin \{4, 6\}$, then $|C| = 5$ and $N_C(v_1) = N_C(v_2)$, by Lemma~\ref{RCedges}.
  Furthemore, Lemma~\ref{RCedges}, implies that there are two adjacent vertices 
  $c, c' \in N_C(v_1)  = N_C(v_2)$, but then
  $v_1cv_2c'v_1$ is a chorded cycle contradicting (O1).
\end{proof}

In the following sections, we will often show that every $C$ in $\Cc$ is a 6-cycle.  Furthermore, it will often be the case that there exists some $u \in R$ such that $\|u,C\| = 3$ for every $C \in \Cc$.  The following lemma will be useful in considering the neighbors of $u$ in $R$ and their adjacencies in $C$.

\begin{lemma}\label{uvnhd}
Let  $C \in \Cc$ with $|C| = 6$, and let $u,v \in R$ such that $uv \in E(G)$.  If $\|u,C\| = 3$ and $\|v,C\| \ge 1$, then $N_C(u) \cap N_C(v) = \emptyset$.
\end{lemma}

\begin{proof}
  By Lemma \ref{RCedges}, we may assume that $A = \{a_1,a_2,a_3\}$ and $B = \{b_1,b_2,b_3\}$ are the partite sets of $G[C] \cong K_{3,3}$ with $N_C(u) = A$.  Suppose on the contrary that $va_1 \in E(G)$.  Then $ua_2b_1a_1vu$ is a 5-cycle with chord $ua_1$.  This contradicts (O1).
\end{proof}

%We now consider the structures of several small graphs, each with no chorded cycle.
%In many arguments of this paper, we carefully count edges between two sets of vertices. The situations below occur frequently, so %we have captured them in lemmas to spare the reader repetition.  
%The proofs of the following lemmas are straightforward, and so are omitted.

\begin{comment}
\begin{lemma}%\label{P2P}
%[Figure~\ref{P2P3F}]
\label{P2P3}
Suppose $H$ is a graph with no chorded cycle. Let $ U \cup W \subseteq V(H)$,
 $U=\{u_1,u_2\}$, $W=\{w_1,w_2,\ldots w_k\}$, and $U \cap W = \emptyset$.
 Suppose $w_1w_2\cdots w_k$ is a path, and $H-W$ contains a $(u_1,u_2)$-path. Then $\|U,W\| \leq 3$,
with equality if and only if for some $i \in [2]$,
%$\|u_i,W\|=2$ and $\|u_{3-i},W\|=1$, with the neighbor of $u_{3-i}$ strictly between the neighbors of $u_i$. Further, if $\|u_i,W\|=2$ and $\|u_{3-i},W\|=1$, then no vertex on the $(u_1,u_2)-$path in $H\setminus W$ has any neighbor in $W$.
% TM added the word "internal"
$\|u_i,W\|=2$ and $\|u_{3-i},W\|=1$, with the neighbor of $u_{3-i}$ strictly between the neighbors of $u_i$. Further, if $\|u_i,W\|=2$ and $\|u_{3-i},W\|=1$, then no internal vertex on the $(u_1,u_2)-$path in $H\setminus W$ has any neighbor in $W$.
\end{lemma}
\end{comment}

\begin{lemma}\label{P2P3}
Suppose $H$ is a graph with no chorded cycle.  Let $U$ and $W$ be two disjoint paths in $H$ and let $u_1$ and $u_2$ be the endpoints of $U$.  Then $\|\{u_1,u_2\},W\| \leq 3$.  If equality holds, then $u_1 \neq u_2$ and for some $i \in [2]$, $\|u_i,W\|=2$ and $\|u_{3-i},W\|=1$, with the neighbor of $u_{3-i}$ strictly between the neighbors of $u_i$ on $W$; in addition, $\|U,W\| = 3$.
\end{lemma}

\begin{proof}
Let $W = w_1w_2\dots w_t$ for some $t \ge 1$. $\|u_1,W\| \le 2$ and $\|u_2,W\| \le 2$, as $H$ does not contain a chorded cycle.  Thus, if $\|\{u_1,u_2\},W\| \ge 3$, we may assume that $u_1 \neq u_2$, and, without loss of generality, that $\|u_1,W\| = 2$ and $\|u_2,W\| \ge 1$.  Suppose $u_1w_i, u_1w_j \in E(H)$ such that $i < j$, and let $u_2w_\ell \in E(H)$ for some $\ell$.

If $\ell \le i$, then $w_\ell Ww_ju_1Uu_2w_\ell$ is a  cycle with chord $u_1w_i$.  If $\ell \ge j$, then $w_iWw_\ell u_2Uu_1w_i$ is a cycle with chord $u_1w_j$.  Thus, the neighbors of $u_2$ in $W$ are internal vertices of the path $w_iWw_j$.  If $\|u_2,W\| = 2$, then suppose $\ell$ is the largest index such that $u_2w_\ell \in E(H)$.  However, $w_iWw_\ell u_2Uu_1w_i$ is a cycle containing a chord incident to $u_2$.  So $\|u_2,W\| = 1$.

Now if $v$ is an internal vertex on $U$ such that $vw_m \in E(H)$, then by replacing $u_2$ with $v$, we deduce that $i \le m \le j$.  If $m \le \ell$, then $w_iWw_\ell u_2Uu_1w_i$ is a cycle with chord $vw_m$, and if $m > \ell$, then $w_\ell Ww_j u_1Uu_2w_j$ is a cycle with chord $vw_m$.  This proves the lemma.
\end{proof}

\section{Suppose $V(R) \neq V(P)$.}\label{R ne P}
%%%%%%%%%%%%%%%%%%%%%%%%%%%%%%%%%%
%%%%%%%%%%%%%%%%%%%%%%%%%%%%%%%%%%
%%%%%%%%%%%%%%%%%%%%%%%%%%%%%%%%%%
In this section, we make the assumption that $V(R) \neq V(P)$.  That is, there exists some vertex $v \in R\setminus P$.  In addition, we will use the convention that $p$ and $p'$ are the endpoints of $P$, and $q$ (resp. $q'$) is the neighbor of $p$ (resp. $p'$) on $P$.  By the maximality of $P$, $vp \notin E(G)$ so that $d_G(v) + d_G(p) \ge 6k - 2$.  Similarly for $v$ and $p'$.  

Our aim is to show that $G = G_1(n,k)$, which is a complete bipartite graph.  To aid us, we define a set of vertices $T := \{v \in R: d_R(v) = 2\}$.  We will
show that $T$ is contained in one of the partite sets of $G_1(n,k)$.

%\subsection{Vertices with $R$-degree at most 2}

\begin{lemma}\label{vp}
If $v \in R\setminus P$, then $\|\{v,p\},C\| \leq 6$ for every $C \in \Cc$, with equality only if 
\begin{enumerate}
\item[\textup{(i)}] $|C| \in \{4,6\}$ and $N_C(v)=N_C(p)$, or 
\item[\textup{(ii)}] $\|p,C\|=|C|=4$. \end{enumerate}
\end{lemma}
\begin{proof}
Suppose $v \in R\setminus P$ and $\|\{v,p\},C\| \geq 6$ for some $C \in \Cc$.  If $\|\{v,p\}, C\| \ge 7$, then either $\|v,C\| = 4$ or $\|p,C\| = 4$, so that $G[C]\cong K_4$ by Lemma~\ref{RCedges}. If $\|v,C\|=4$, then $\|p,C\|=0$, lest we extend $P$ by adding a neighbor of $p$ in $C$, and replace said neighbor in $C$ with $v$, violating (O3). If $\|p,C\|=4$, then $\|v,C\| \leq 2$, else there exists $c \in C$ such that $C - c + v \cong K_4$, and we can extend $P$ by adding $c$, violating (O3).  So, $\|\{v,p\},C\|\le 6$, and if equality holds, then either (ii) occurs, or $\|v,C\|=\|p,C\|=3$.  We may assume $\|v,C\|= \|p,C\| = 3$, so that $|C| \in \{4,5,6\}$ by Lemma~\ref{RCedges}.

By Lemma~\ref{noC5}, $|C| \in \{4,6\}$.  Suppose $|C|=4$ and $\|v,C\|=\|p,C\|=3$. Note that $G[N_C(v)\cup \{v\}]$ forms a chorded 4-cycle with at least the same number of chords as $C$. If $p$ is adjacent to the vertex in $V(C)\setminus N_G(v)$, we use that vertex to extend $P$, violating (O3). So (i) holds.

Finally, suppose $|C|=6$. By Lemma~\ref{RCedges}, if $v$ and $p$ do not have the same neighborhood, they are adjacent to disjoint sets of vertices, and $C+p$ and $C+v$ both contain $K_{3,4}$. In this case, we extend $P$ using any $c \in N_C(p)$, and replace $C$ with a chorded cycle in $C - c+v$. This violates (O3), so (i) holds.
\end{proof}

\begin{lemma}\label{pbuds}
For any $v \in R\setminus P$,
$\|\{v,p\},R\| \geq 4$, so that $\|v,R\| \geq 2$.  Moreover, $|P| \geq 3$. 
\end{lemma}

\begin{proof}
Let $v \in R\setminus P$.  By the maximality of $P$, $pv \notin E(G)$.  Thus, by Lemma~\ref{vp}, 
\[2(3k-1) \leq d_G(v)+d_G(p)=\|\{v,p\},\Cc\|+\|\{v,p\},R\| \leq 6(k-1)+\|\{v,p\},R\|,\] 
so $\|\{v,p\},R\| \geq 4$.
Since $\|p,R\| \leq 2$, it follows that $\|v,R\| \geq 2$. 
Then $v$ and two of its neighbors form a path of length three in $R$, hence $|P| \geq 3$.
\end{proof}

\begin{lemma}\label{R-Pv1v2}
For any maximal path $P'$ in $R\setminus P$, label the (not necessarily distinct) endpoints $v_1$ and $v_2$ so that $\|v_1,P\| \leq \|v_2,P\|$. Then:
\begin{enumerate}[label=\textup{(\alph*)}, ref={\ref{R-Pv1v2}\alph*}]
\item$\|v_2,P\|\le2$, and if $v_1 \neq v_2$ then $\|v_1,P\|\le1$, 
\item\label{comp T} $d_R(v_1) = 2$ (this implies $v_1 \in T\setminus V(P)$ so that $T\setminus V(P) \neq \emptyset$), and
\item if $\|v_2,P\|=2$  and $\|v_1,P\|=1$, then $\|P'-v_1-v_2,P\|=0$.
\end{enumerate}
\end{lemma}

\begin{proof}
Since $R$ contains no chorded cycle, no vertex in $R\setminus P$ has three neighbors in $P$, so $\|v_2,P\| \leq 2$.
%TM - doesn't everything just follow from Lemma P2P3?
%By Lemma~\ref{P2P3},
%no two vertices sharing a component in $R\setminus P$ each have two neighbors in $P$, 
%and if $\|\{v_1,v_2\},P\|\geq 3$ then $\|\{v_1,v_2\},P\|=3$, $v_1 \neq v_2$, and $\|P'-v_1-v_2,P\|=0$. This gives us~(a) and (c).
Lemma~\ref{P2P3} then gives (a) and (c).

It remains to show (b). If $\|v_1,P\|=0$, then using Lemma~\ref{pbuds} and the maximality of $P'$, $d_R(v_1) = \|v_1,P'\|=2$.
If $v_1=v_2$, then $\|v_1,R\|=\|v_2,P\|=2$.  So suppose $\|v_1,P\|=1$ and $v_1 \neq v_2$. Since $\|v_2,P\| \geq \|v_1,P\| = 1$, there exist $a_1, a_2 \in P$
 (perhaps $a_1=a_2$) such that $v_1a_1,v_2a_2 \in E(G)$.
Then $v_1P'v_2a_2Pa_1v_1$ is a cycle.  Since it has no chord, $\|v_1,P'\|=1$, so $\|v_1,R\|=2$ and $v_1 \in T$.
\end{proof}

\begin{lemma}\label{pvnhd}
$d_R(p) = d_R(p') = 2$.  Additionally, for every $v \in T\setminus V(P)$ and every $C \in \Cc$: 
\begin{enumerate}[label=\textup{(\alph*)}, ref={\ref{pvnhd}\alph*}]
%\item\label{pvnhda} $d_R(p) = 2$,
\item\label{pvnhdd} $|C| \in \{4,6\}$,
\item\label{pvnhdb} $\|p,C\|=3$, and
\item\label{pvnhdc} $N_\Cc(v)=N_\Cc(p)$.
\end{enumerate}
\end{lemma}

\begin{proof}
By Lemma~\ref{R-Pv1v2}, $v \in T\setminus V(P)$ exists so that $d_R(v) = 2$.  Lemma \ref{pbuds} implies $\|\{v,p\},R\| \ge 4$, and hence,  $d_R(p) = 2$ and $\|\{v,p\},R\| = 4$.  Since $vp \notin E(G)$,  $\|\{v,p\},\Cc\| \geq (6k-2)-4= 6(k-1)$. By Lemma~\ref{vp}, $\|\{v,p\},C\|=6$ for all $C \in \Cc$.  If we can show that $\|p,C\| = 3$ for all $C \in \Cc$, then we are done by Lemma \ref{vp}.

%If not, then by Lemma \ref{vp}, $\|p,C\| = 4$ and $G[C] \cong K_4$ by Lemma \ref{RCedges}.  Thus, $\|v,C\| = 2$, and by 
If not, then there exists $C \in \Cc$ such that $\|p, C\| > 3$, so $\|p,C\| = 4$ and $G[C] \cong K_4$ by Lemma \ref{RCedges}.  Thus, $\|v,C\| = 2$, and by 
%and since $|P| \ge 3$, a symmetric argument implies $\|p',C\|= 4$.  
 Lemma~\ref{vp}, there exists $u \in N_C(p')$.   Since $\|p,P\| = 2$, $P + u$ forms a chorded cycle, so since $C - u + v$ also forms chorded cycles, we have a contradiction.  Thus, $\|p,C\| = 3$ as desired.
\end{proof}

From Lemma \ref{pvnhd} we immediately obtain the following.

\begin{cor}\label{mindeg}
$d_G(p) = d_G(p') = 3k - 1$, and consequently, $d_G(v) \ge 3k - 1$ for all $v \in R\setminus P$.
\end{cor}

Recall that $\Pp$ is the set of vertices in $P$ that are the endpoint of a path spanning $V(P)$.  Note Lemmas \ref{vp}, \ref{pbuds}, \ref{R-Pv1v2}, and \ref{pvnhd} apply to each $p^* \in \Pp$.  Thus, $\Pp \subseteq T$, and furthermore, for all $p^*_1, p^*_2 \in \Pp$, $N_\Cc(p^*_1) = N_\Cc(p^*_2)$.

%\subsection{Structure of $G$}

\begin{lemma}\label{|C|=6}
For every $C \in \mathcal C$, $G[C] \cong K_{3,3}$.
\end{lemma}
\begin{proof}
  If not, by Lemma~\ref{RCedges} and Lemma~\ref{pvnhd}, we may assume that there exists $C \in \Cc$ with $|C| = 4$.  Suppose $V(C) = \{c_1,c_2,c_3,c_4\}$.   Let $v \in T\setminus V(P)$, which we know exists by Lemma \ref{R-Pv1v2}.  By Lemmas \ref{RCedges} and \ref{pvnhd}, we may assume that $N_C(p) = N_C(p') = N_C(v) = \{c_1,c_2,c_3\}$ and $c_2c_4 \in E(G)$.  Since $\|p,P\| = 2$ by Lemma \ref{pvnhd}, $P+c_1$ and $C - c_1 + v$ contain chorded cycles,  a contradiction.  
%The statement then follows from Lemmas \ref{RCedges} and \ref{pvnhd}.
\end{proof}

For the remainder of this section, we will use the fact that 
%By Lemmas \ref{RCedges}, \ref{pvnhd}, and \ref{|C|=6}, we may assume that 
for each $C \in \Cc$, $G[C] \cong K_{3,3}$ and,
that there exist $A \subseteq C$ such that $A$ is
a partite set of $C$ and such that, 
for every $p^* \in \mathcal{P}$, $N_C(p^*) = A$, 
%$A$ is the same
%partite set,
%the neighborhood
%of every $p^* \in \mathcal{P}$ in $C$ is the same partite set of $C$
%with $N_C(p)$ consisting of one partite set of $C$ 
without mentioning Lemmas \ref{pvnhd}, and \ref{|C|=6}.

\begin{lemma}\label{nhdofv}
For every $C \in \Cc$, if $v \in R\setminus P$ has a neighbor in $C$, then $N_C(v) \subseteq N_C(p)$, unless $|N_C(v)| = 1$.
\end{lemma}

\begin{proof}
Fix $C \in \Cc$, and let $A = \{a_1,a_2,a_3\}$ and $B = \{b_1,b_2,b_3\}$ be the partite sets of $C$ such that $N_C(p) = N_C(p') = A$.  Suppose on the contrary, there exists  $v \in R\setminus P$ with $|N_C(v)| \ge 2$ such that, say $vb_3 \in E(G)$.  

By Lemma \ref{pvnhd}, $\|p,P\| = 2$ so that $P +a_i$ contains a chorded cycle for each $i \in [3]$.  If $vb_2 \in E(G)$, then $vb_3a_3b_1a_2b_2v$ is a cycle with chord $a_2b_3$.  However, $P+a_1$ also contains a chorded cycle, a contradiction. 

So we may assume that $va_3 \in E(G)$.  However, $vb_3a_2b_2a_3v$ is a 5-cycle with chord $a_3b_3$ contradicting (O1).  Thus, $N_C(v) \subseteq A = N_C(p)$, as desired.
\end{proof}

\begin{lemma}\label{R-P=T}
$R\setminus P$ is an independent set, and $V(R\setminus P) \subseteq T$.
\end{lemma}
\begin{proof}
% TM - changed this proof to use Lemma~\ref{P2P3}
Suppose $R\setminus P$ is not an independent set.  Then there exists a maximal path $P'$ in $R\setminus P$ with distinct endpoints $v_1$ and $v_2$, labeled as in Lemma \ref{R-Pv1v2}.   Thus,$\|v_2,P\| \le 2$, and, hence, $d_R(v_2) \le 4$.  Since $pv_2 \notin E(G)$, Lemma \ref{mindeg} implies that $d_G(v_2) \ge 3k - 1 > 4$,  which implies that there exists $C \in \Cc$ such that
$v_2$ has a neighbor $c \in C$.
%Fix $C \in \Cc$, and 

Let $A = \{a_1,a_2,a_3\}$ and $B = \{b_1,b_2,b_3\}$ be the partite sets of $C$ such that $N_C(p) = N_C(p') = A$.
By Lemmas \ref{R-Pv1v2} and \ref{pvnhd}, $v_1 \in T\setminus V(P)$ and $N_C(v_1) = A$.
We can assume $a_1 \neq c$, so that there exists a path $W$ in $C - a_1$ that contains $a_2$ and $a_3$ for which $c$ is an endpoint.  Since $\|v_1, W\| \ge 2$ and $v_2$ is adjacent to an endpoint of $W$, $\|\{v_1,v_2\}, W\| \ge 3$ and 
 Lemma \ref{P2P3} implies there is a chorded cycle in $G[V(P') \cup V(C - a_1)]$.  However, as $\|p,P\| = 2$, $P+a_1$ also contains a chorded cycle, a contradiction.
%Suppose $R\setminus P$ is not an independent set.  Then there exists a maximal path $P'$ in $R\setminus P$ with distinct endpoints $v_1$ and $v_2$, labeled as in Lemma \ref{R-Pv1v2}.  Fix $C \in \Cc$, and let $\{a_1,a_2,a_3\}$ and $\{b_1,b_2,b_3\}$ be the partite sets of $C$ such that $N_C(p) = N_C(p') = \{a_1,a_2,a_3\}$.
%
%By Lemma \ref{R-Pv1v2}, $v_1 \in T$ so that by Lemma \ref{pvnhd}, $N_C(v_1) = \{a_1,a_2,a_3\}$.  We claim that $v_2$ has at most one neighbor in $C$.  If not, then by Lemma \ref{nhdofv}, $N_C(v_2) \subseteq \{a_1,a_2,a_3\}$.  Without loss of generality, assume $v_2a_2, v_2a_3 \in E(G)$.  Then $P'+a_2+b_3+a_3$ and $P+a_1$ each contain a chorded cycle, a contradiction.  This proves the claim.
%
%As this argument holds for all cycles in $\Cc$, $d_R(v_2) \ge 3k - 1 - (k - 1) = 2k \ge 4$.  Since $R$ has no chorded cycles, $d_R(v_2) \le 4$ so that $\|v_2,P'\| \ge 2$ and $\|v_2,C\| \ge 1$.  If $v_2a_3 \in E(G)$, then $P+a_1$ and $P'+a_3$ each contain a chorded cycles.  If $v_2b_3 \in E(G)$, then $P+a_1$ and $P'+a_3+b_3$ each contain a chorded cycle.  In either case, we obtain a contradiction.  Thus, $R\setminus P$ is an independent set.
%
%
%Now by Corollary \ref{mindeg} and Lemma \ref{pvnhd}, $d_R(v) \ge (3k - 1) - 3(k - 1) =2$ for all $v \in R\setminus P$.  Thus, $d_R(v) = 2$ and $v \in T$, as desired.
\end{proof}

%We now define the sets that will make up most of our partite sets.  
Let $\Ss := N_\Cc(p)$, and let $\Tt := ((\bigcup_{C \in \Cc} V(C))\setminus \Ss) \cup T$.

\begin{prop}\label{bipar_sub}
$G[\Ss \cup \Tt] \cong K_{3k-3,|\Tt|}$, and no vertex in $G$ has neighbors in both $\Ss$ and $\Tt$.
\end{prop}

\begin{proof}
  By Lemma~\ref{|C|=6}, $\mathcal C$ consists of $k-1$ copies of $K_{3,3}$. Lemmas \ref{pvnhd} and \ref{R-P=T} tell us that, for every $v \in R\setminus P$, $N_\Cc(v)=\Ss$. Given $C \in \Cc$, $a \in V(C) \cap \Tt $, and $v \in R\setminus P$, we can create a chorded cycle $C'$ by swapping $a$ and $v$ in $C$. Note $G[C'] \cong K_{3,3}$, and we have not changed any vertices in $P$.  Then replacing $C$ with $C'$ in $\Cc$ results in a collection of $k-1$ chorded cycles satisfying (O1) through (O3). Thus all the previous lemmas apply, and, in particular, Lemma~\ref{R-Pv1v2} and Lemma~\ref{R-P=T} imply that $a \in T$.  So by Lemma~\ref{pvnhd}, and the fact that $N_C(a)= V(C) \cap \Ss$, we conclude $N_\Cc(a)=\Ss$.  Hence, every vertex in $\Tt$ is adjacent to every vertex in $\Ss$, and $G[\Ss \cup \Tt]$ contains a copy of $K_{|\Ss|,|\Tt|}$.

We claim $G[\Ss \cup \Tt]$ has no additional edges.  Note $|\Tt|>3(k-1)$ and $|\Ss|=3(k-1)$. If there exists any edge with both endpoints in $\Tt$, or both endpoints in $\Ss$, then we find a set of $k-1$ chorded cycles, $k-2$ of which are 6-cycles, and one of which is a 4-cycle, violating (O1). So $G[\Ss \cup \Tt]\cong K_{|\Ss|,|\Tt|} \cong K_{3k-3,|\Tt|}$. 

If any vertex of $V(G)\setminus(\Ss \cup \Tt)$ has  neighbors in both $\Ss$ and $\Tt$, then in a similar manner, we find $k-1$ disjoint chorded cycles, one of which is a 5-cycle and the rest of which are 6-cycles, again violating (O1).
\end{proof}

Recall that $q$ and $q'$ were defined as the neighbors of $p$ and $p'$, respectively, on $P$.  Since $\|p,P\| = 2$ by Lemma \ref{pvnhd}, there exists $w \in N_R(p)\setminus\{q\}$.  As a consequence of Proposition~\ref{bipar_sub}, $w \neq p'$.   Now the neighbor of $w$ on $pPw$ is the endpoint of a path that spans $V(P)$.  Thus, $|\Pp| \ge 3$.

\begin{lemma}\label{|P|=3}
$|\Pp|=3$
\end{lemma}

\begin{proof}
Suppose $|\Pp|\geq 4$, with $p_1,p_2,p_3,p_4$ the first four members of $\Pp$ along $P$.  In particular, $p_1 = p$.  Fix $C \in \Cc$, and let $A = \{a_1,a_2,a_3\}$ and $B = \{b_1,b_2,b_3\}$ be the partite sets of $C$ such that $N_C(p_i) = A$ for each $i \in [4]$.  

By Lemma \ref{uvnhd}, $N_C(q) \subseteq B$.  So in particular, $q \neq p_2$.  If $q$ has a neighbor in $C$, say $b_1$, then $qb_1a_1b_2a_2p_1q$ is a 6-cycle with chord $p_1a_1$ and $p_2Pp_4a_3p_2$ is a cycle with chord $p_3a_3$, a contradiction. 

So we may assume that for every $C \in \Cc$, $N_C(q) = \emptyset$.  That is, $\|q,R\| = d_G(q)$. Since $\|p_3,P\|=2$ by Lemma~\ref{pvnhd},
 $q$ is not adjacent to $p_3$. 
Then since $d_G(p_3)=3k-1$ by Corollary~\ref{mindeg},  $d_G(q) \ge 3k -1 \ge 5$.  Since $\|q,P\| \le 3$, $q$ must be adjacent to two vertices $v_1,v_2 \in R\setminus P$.  By Lemma \ref{pvnhd}, $N_C(v_1) = N_C(v_2) = A$.  However, this yields the cycles $v_1qv_2a_2b_1a_1v_1$ and $p_2Pp_4a_3p_2$ with chords $v_1a_2$ and $p_3a_3$, respectively, a contradiction.
\end{proof}

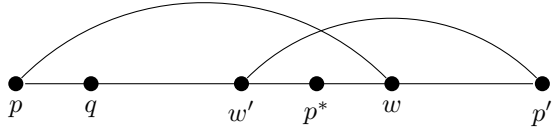
\begin{figure}[ht]\centering
\begin{tikzpicture}
\foreach \x in {1,2,4,5,6,8}
	{\draw (\x,0) node(a\x){};}
\draw (a1)--(a8);
\draw (a1) node[vertex,label=below:$p$](p){};
\draw (a2) node[vertex,label=below:$q$](q){};
\draw (a4) node[vertex,label=below:$w'$](w'){};
\draw (a5) node[vertex,label=below:$p^*$](p^*){};
\draw (a6) node[vertex,label=below:$w$](w){};
\draw (a8) node[vertex,label=below:$p'$](p'){};
\draw(p) to[out=45, in=135] (w);
\draw(w') to[out=45, in=135] (p');
\end{tikzpicture}
\caption{Setup for Lemma \ref{K23}}\label{pstar}
\end{figure}

\begin{lemma}\label{K23}
$G[P] \cong K_{2,3}$
\end{lemma}
\begin{proof}
By Lemma \ref{|P|=3}, we may assume that $\Pp = \{p,p',p^*\}$.  Recall
that $\Pp \subseteq \Tt$, so Lemma~\ref{bipar_sub} implies that
$\Pp$ is an independent set.  Lemma~\ref{pvnhd} implies that  
$\|p,P\| = \|p',P\| = \|p^*, P\| = 2$, so there exist $w$ and $w'$ on  $P$ such that $w \neq q$ and $w' \neq q'$ and $N_P(p) = \{q, w\}$ and $N_P(p') = \{q', w'\}$.  Furthermore, since $|\Pp| = 3$, and both
the neighbor of $w$ on $pPw$ and the neighbor of $w'$ on $w'Pp'$ are
in $\Pp$, we can conclude that $w \neq w'$ and $N_P(p^*) = \{w, w'\}$, 
i.e.\ $wPw'$ is the path on three vertices $wp^*w'$.

Since $G[P]$ does not contain a chorded cycle, $qq' \notin E$,
so if $w=q'$ and $w'=q$, then $G \cong K_{2,3}$.
So if $G \not \cong K_{2,3}$, then without loss of generality we can assume that $q \neq w'$ as in Figure \ref{pstar}.   Thus, $qp', pw' \notin E(G)$  so, by Corollary~\ref{mindeg}, $d_G(q), d_G(w') \ge 3k - 1$.

Fix $C \in \Cc$ with partite sets $A = \{a_1,a_2,a_3\}$ and $B = \{b_1,b_2,b_3\}$ such that $N_C(p) = N_C(p') = N_C(p^*) = A$.  By Lemma \ref{uvnhd}, $N_C(q) \subseteq B$ and $N_C(w') \subseteq B$.

Since $d_G(q) \ge 3k - 1$, $\|q,C \cup R\| \ge (3k - 1) - 3(k - 2) = 5$.  Also,   $\|q, P\| \le 3$. This holds for $w'$ as well.  Thus, both $q$ and $w'$ have two neighbors in $B \cup (R\setminus P)$.  Let $v_1$ and $v_2$ be distinct vertices in $B \cup (R\setminus P)$ such that  $v_1q, v_2w' \in E(G)$.  We may assume that $v_2 \neq b_3$.  Observe that $N_C(v_1) = N_C(v_2)= A$.  Then the cycle $pqv_1a_1b_3a_3p$ has chord $pa_1$, and the cycle $w'v_2a_2p'Pw'$ has chord $a_2p^*$, a contradiction.
%
%Thus, the neighbors of $p$ on $P$ are $w$ and $w'$.  By symmetry, the same holds for $p'$ so that $G[P]$ contains $K_{2,3}$.  Since the addition of any other edge creates a chorded cycle in $R$, this proves the lemma.
\end{proof}

\begin{lemma}
$G=G_1(n,k)$
\end{lemma}
\begin{proof}
By Lemma \ref{K23}, let $\{p_1,p_2,p_3\}$ and $\{q_1,q_2\}$ denote the partite sets of $G[P]$.  Recall that $\Pp \subseteq T$ so that  $G[\Ss \cup \Tt]$ contains every vertex of $G$ except for $q_1$ and $q_2$.  

By  Lemmas \ref{pvnhd} and \ref{R-P=T} and Corollary \ref{mindeg}, $\|v,P\| = 2$ for all $v \in R\setminus P$, and by Proposition~\ref{bipar_sub}, $N_R(v) = \{q_1,q_2\}$.  Since $\Tt$ is an independent set in $G$, for each $u \in \Tt\setminus T$, $\|u,R\| \ge (3k - 1) - 3(k - 1) = 2$.  Thus, $uq_1, uq_2 \in E(G)$, and so $N_G(q_i) \supseteq \Tt$ for $i \in [2]$.  That is, $G \supseteq K_{|\Ss|+2,|\Tt|}=K_{3k-1,|G|-3k+1}=G_1(n,k)$. Since adding any edge to $G_1(n,k)$ results in a graph with $k$ disjoint chorded cycles, we conclude $G=G_1(n,k)$.
\end{proof}

%%%%%%%%%%%%%%%%%%%%%%%%%%%%%%%%%%
%%%%%%%%%%%%%%%%%%%%%%%%%%%%%%%%%%
%%%%%%%%%%%%%%%%%%%%%%%%%%%%%%%%%%
\section{Suppose $V(R)=V(P)$} \label{long path}
%and $|P| \geq 6$}\label{long path}
%%%%%%%%%%%%%%%%%%%%%%%%%%%%%%%%%%
%%%%%%%%%%%%%%%%%%%%%%%%%%%%%%%%%%
%%%%%%%%%%%%%%%%%%%%%%%%%%%%%%%%%%

In this section, we assume $V(P)=V(R)$.  Since adding any edge to $G$ results in $k$ chorded cycles, by (O1) $|P| \geq 4$.  If $|P| \ge 6$, we label $P=p_1q_1r_1\cdots r_2q_2p_2$.
Note that, since $G[R]$ has no chorded cycles,
for every $v \in R$, $\|r, R\| \le 4$.
When $|P| = 5$, we let $P = p_1q_1rq_2p_2$, and when
$|P| = 4$, we let $P = p_1q_1q_2p_2$.
We call an edge in $E(G[P])\setminus E(P)$ a \emph{hop}.  
If $Q = v_1 \dotsm v_{|R|}$ is a spanning path of $R$, then we call an edge $v_iv_j$ a \emph{hop (on $Q$)} if $|i - j| > 1$.

\begin{lemma}\label{hops}
  If $Q = v_1 \dotsm v_{|R|}$ is a spanning path of $R$ and $v_iv_j$ is a hop
  with $i < j$, 
  then $v_{i+1}$ and $v_{i+2}$ cannot both be incident to hops, and
  similarly, $v_{j-1}$ and $v_{j-2}$ cannot both be incident to hops.
\end{lemma}
\begin{proof}
%Without loss of generality, assume $i < j - 1$.
  Suppose that, on the contrary, $v_{i+1}v_k$ and $v_{i+2}v_{k'}$ are both hops.
  Note that, if we consider only the hop $v_iv_j$ and the hop $v_{i+1}v_k$,
  $v_jv_iv_{i+1}Qv_j$ is a chorded cycle if $i+3 \le k \le j$, 
  and $v_kQv_iv_jQv_{i+1}v_k$ is a chorded cycle if $k \le i-1$,
  so $k > j$.
  Repeating this argument but now only considering
  the hops $v_{i+1}v_k$ and $v_{i+2}v_{k'}$
  gives us that $k' > k$, but then
  $v_iv_{i+1}v_{i+2}v_{k'}Qv_jv_i$ is a cycle with chord $v_{i+1}v_k$,
  a contradiction.
  By symmetry, the lemma holds.
\end{proof}

\begin{lemma}\label{pchords}
  For any $p \in \mathcal{P}$, $d_R(p) = 2$ unless $R$ is a path.
\end{lemma}
\begin{proof}
  Let $v_1 \dotsm v_{|R|}$ be  a spanning path in $R$, and let $p = v_1$.  
  Assume $d_R(p) = 1$, and that $R$ is not a path.  
  Since $R$ is not a path, hops exist. Let $v_iv_j$, $i < j$, be a hop such that 
 for all $k, j < k \le |R|$, $v_k$ is not incident to a hop. 
  Note, because $d_R(p) = 1$,  that $i \neq 1$.
  %Thus, $R$ contains `hops' (i.e., edges of the form $v_iv_j$ where $|i - j| > 1$).  Let $v_i \in V(P)$ be the last vertex on $P$ incident to a hop.  That is, for all $j, i < j \le |R|$, $v_j$ is not incident to a hop.

  Let $D$ be the cycle $v_jv_iv_{i+1} \dotsm v_{j-1}v_j$.
  Since $R$ contains no chorded cycles, 
  $v_j$ is incident to exactly one hop
  and $v_{j-1}$ is incident to at most one hop.  
  If $v_{j-1}$ is not incident to a hop let $x=v_{j-1}$ and $y=v_{j}$,
  and if $v_{j-1}$ is incident to exactly one hop, let $x = v_{j-2}$ and $y = v_{j-1}$.
  By Lemma~\ref{hops}, when $v_{j-1}$ is incident to a hop, $v_{j-2}$ is not incident to a hop, 
  so in either case, 
  $xy \in E(D)$, $d_R(x) + d_R(y) \le 5$, and $px,py \notin E(G)$.
  Therefore, 
  \begin{equation*}
    2\|p, \mathcal C\| + \|\{x,y\}, \mathcal C\| \ge 2(6k - 2) - (2\|p, R\| + \|\{x,y\}, R\|) > 12(k - 1).
  \end{equation*}
  So there exists $C \in \mathcal C$ such that $2\|p, C\| + \|\{x,y\}, C\| \ge 13$.  Thus, $\|v,C\| = 4$ for some $v \in \{p,x,y\}$, and
  by Lemma~\ref{RCedges}, $G[C] \cong K_4$. Further, $\|\{x,y\}, C\| \ge 5$ so that there exists
  $c \in C$ such that $xc,yc \in E(G)$ and $D + c$ contains a chorded cycle. 
  Also $2\|p, C\| \ge 5$, which implies $\|p, C - c\| \ge 2$ so that $C - c + p$ contains a chorded cycle, a contradiction.
\end{proof}

\begin{lemma}\label{R6toC}
  If $|R| \ge 6$, then there exists $F^+ \subseteq V(R)$ such that $|F^+| = 6$ and such that
  for every $C \in \Cc$ and every pair of distinct vertices $u, u' \in F^+$, 
  $\|\{u,u'\}, C\| \ge 1$.
\end{lemma}
\begin{proof}
  First we find $F^+ \subseteq V(R)$ such that $\|F^+, R\| \le 15$. If $R$ is a path, this is trivial, so we assume $R$ has at least one hop. 
%  Suppose $d_R(r_i) \le 3$ for some $i \in [2]$.  
  By Lemmas \ref{hops} and \ref{pchords}, $p_i$ is incident to a hop so that $q_i$ and $r_i$ cannot both be incident to hops.  
  If $d_R(r_i) \le 3$ for some $i \in [2]$, then since $d_R(q_i) \le 3$ and $d_R(p_i) = 2$ by Lemma \ref{pchords}, $\|\{p_i,q_i,r_i\}, R\| \le 7$.  If $d_R(r_i) = 4$, then $d_R(p_i) = d_R(q_i) = 2$, so that $\|\{p_i,q_i,r_i\},R\| \le 8$.  Therefore, $F^+ := \{p_1, q_1, r_1, r_2, q_2, p_2\}$ suffices when either $d_R(r_1) \le 3$ or  $d_R(r_2) \le 3$. 
  In this case, we let $r^*_1 = r_1$.

  When $d_R(r_1) = d_R(r_2) = 4$, $|R| \ge 7$, since $R$ has no chorded cycles, and there exists a vertex $u$ following $r_1$ on $P$ with $d_R(u) \le 3$.  Here, we let $F^+ := \{p_1,q_1, u, r_2, q_2,p_2\}$.
  and let $r^*_1 = u$.
%  Define $r^*_1 \in \{r_1,u\}$ such that $r^*_1 \in F^+$.  
  Thus,  in both cases, $F^+ = \{p_1,q_1,r^*_1, r_2, q_2,p_2\}$.
  
We claim that we can partition $F^+$ into three sets so that each set will consist of two nonadjacent vertices. Define $F_1:=\{p_1,q_1,r_1^*\}$
and $F_2:=\{p_2,q_2,r_2\}$, and let $H$ be the subgraph of $G$ on the vertex set $F^+$ containing precisely those edges of $G$ with one endpoint in $F_1$ and the other in $F_2$.
Because $R$ contains no chorded cycle, every vertex in $F_2$ has at most two neighbors in $F_1$, and vice-versa. That is, $H \subseteq 3K_2$. Therefore we can label $F_1=\{f_1,f_2,f_3\}$ so that 
$f_1p_2, $ $f_2q_2$, and $f_3r_2 $ are all nonedges.

%We now partition $F^+$ into three sets so that each set will consist of two nonadjacent vertices.  
%  If $p_1p_2 \in E(G)$, then we use the pairs $\{p_1,r^*_1\}, \{q_1,q_2\}, \{r_2,p_2\}$.   
%  If $p_1p_2 \notin E$, and additionally $q_1q_2 \in E(G)$  or $r^*_1r_2$ is a hop,
%   then we use $\{p_1,p_2\}, \{q_1,r_2\}, \{r^*_1,q_2\}$.  
%  If $p_1p_2, q_1q_2, r^*_1r_2 \notin E(G)$, then use $\{p_1,p_2\}, \{q_1,q_2\}, \{r^*_1,r_2\}$.
  
  Therefore, $\|F^+, \Cc\| \ge 3(6k - 2) - 15 = 18(k-1) - 3$.  Suppose there exists $C \in \Cc$ for which $\|F^+, C\| \le 14$ so that there exists $C' \in \Cc$ such that $\|F^+, C'\| \ge 19$.  If we can find $v_1, v_2 \in F^+$ such that $\|\{v_1,v_2\}, C' \| \le 6$, then $\|F' - v_1 - v_2, C'\| \ge 13$, contradicting Lemma \ref{C4bounds}.  So for $F^+ = \{v_1,v_2, \dots, v_6\}$, $\|\{v_i, v_{i+1}\}, C'\| \ge 7$ for $i \in \{1,3,5\}$.  However this implies $\|\{v_1,v_2,v_3,v_4\}, C'\| \ge 14$, a contradiction to Lemma \ref{C4bounds}.  

Thus, $\|F^+, C\| \ge 15$ for every $C \in \Cc$.
  If there exists a pair of distinct vertices $u, u' \in F^+$ such that
  $\|\{u,u'\}, C\|  = 0$, then $\|F^+ - u - u', C\| \ge 15$, again a violation of Lemma~\ref{C4bounds}.
\end{proof}

\begin{lemma}\label{struct1}
There exists $F \subseteq V(R)$ such that $p_1, p_2 \in F$, $|F| = 4$ and 
  \begin{enumerate}[label=\textup{(\alph*)}, ref={(\alph*)}]
      \item \label{Fdeg} $\|F, \Cc\| \ge 12(k-1) - 2$ if $R \cong K_{2,3}$, $\|F, \Cc\| \ge 12(k-1) + 2$ if $R$ is a path, 
	and $\|F, \Cc\| \ge 12(k-1)$ otherwise, and 
      \item \label{paths} if $R$ is not a path, then for every $u \in F$, there exists a path $Q$ in $R - u$ such that $F - u \subseteq V(Q)$.
  \end{enumerate}
\end{lemma}

\begin{proof}
  If $R$ is a path or $R \cong K_{2,3}$,  let $F := \{p_1, q_1, q_2, p_2\}$.   When $R$ is a path, $\|F, R\| = 6$, and $p_1q_2, p_2q_1 \notin E(G)$; when $R \cong K_{2,3}$, $\|F, R\| = 10$, and $p_1p_2, q_1q_2 \notin E(G)$.  In both cases,  \ref{Fdeg} and \ref{paths} hold.

So we assume $R \not\cong K_{2,3}$ and $R$ is not a path.
By Lemma~\ref{pchords}, for $i \in [2]$, $\|p_i,P\| = 2$.  Thus, $p_i$ has a neighbor $w_i \in P - q_i$.  
  Let $t_i$ denote the neighbor of $w_i$ on $w_iPp_i$.  Observe that $t_i \in \Pp$, so by Lemma~\ref{pchords}, $\|t_i, P\| = 2$. 
  Suppose $t_1 \neq t_2$, and,
  in this case, let $F:=\{p_1,t_1,t_2,p_2\}$.
  Then $F \subseteq \mathcal P$, so \ref{paths} holds and $\|F,R\| \le 8$. If either $p_1t_1,p_2t_2 \not\in E(G)$ or $p_1t_2, p_2t_1 \not\in E(G)$, then \ref{Fdeg} holds. Suppose (say) $p_1t_1 \in E(G)$. Then $t_1=q_1$, and $t_1p_2 \not\in E(G)$. Then $w_2 \not\in \{p_1,t_1\}$, hence $t_2 \not\in \{t_1,w_1\} = N_R(p_1)$, so also $p_1t_2 \not\in E(G)$. So in this case also, \ref{Fdeg} holds.
%  and  note that
%  if $p_1p_2 \in E(G)$, then $p_1t_1,p_2t_2 \notin E(G)$, and
%  if $p_1p_2 \notin E(G)$, then $t_1t_2 \notin E(G)$.  In either case, $F := \{p_1, t_1, t_2, p_2\}$ suffices as $\|F, R\| \le 8$, so that \ref{Fdeg} holds, and clearly, \ref{paths} holds.

  So assume $t_1 = t_2$, which implies $\|u, P\| = 2$ for all $u \in V(P) - w_1 - w_2$, as otherwise $R$ contains a chorded cycle.
  Also, when $t_1 = t_2$, we may assume that $q_1 \neq w_2$ since $R$ is not isomorphic to $K_{2,3}$.
  In this case, let $F : = \{p_1, q_1, t_1, p_2\}$ and note that $p_1t_1, q_1p_2 \notin E(G)$.  Since $d_R(u) = 2$ for all $u \in F$, \ref{Fdeg} holds.  Since $t_1 = t_2$, $p_1w_1t_1w_2p_2$ is a path in $R- q_1$ containing $F - q_1$ 
  and $F - q_1 \subseteq \Pp$, \ref{paths} holds.
%  It is easy to find such paths for $u \in F \cap \Pp$ so that \ref{paths} holds.
\end{proof}

\begin{cor}
$R$ is not a path.
\end{cor}

\begin{proof}
Let $F \subseteq V(R)$ be as guaranteed in Lemma \ref{struct1}.  If $R$ is a path, then $\|F,\Cc\| \ge 12(k - 1) + 2$, so that there exists $C \in \Cc$ such that $\|F, C\| \ge 13$, which  violates Lemma~\ref{C4bounds}. So $R$ is not path.
\end{proof}

\begin{lemma}\label{struct2}
  Let $F \subseteq V(R)$  be as guaranteed in Lemma \ref{struct1}.  If $\|F, C\| = 12$ for any $C \in \Cc$, then $G[C] \cong K_{3,3}$.
\end{lemma}

\begin{proof}
Let $F \subseteq V(R)$ be as guaranteed in Lemma \ref{struct1} and let $C \in \Cc$.  Suppose that $\|F, C\| = 12$.
  By Lemmas~\ref{C4bounds} and \ref{struct1}, this is true for all $C \in \Cc$,
  unless $R \cong K_{2,3}$.
  By Lemmas~\ref{RCedges} and \ref{noC5}, $C \cong K_{3,3}$ unless $|C| = 4$, 
  so assume $|C| = 4$.
  Note that for any $u \in F$ and $c \in C$, if 
  $C - c + u$ is a chorded cycle, then $\|c, F - u\| \le 2$,
  because there exists a path $Q$ in $R$ such that $F - u \subseteq V(Q)$
  and $G[Q + c]$ cannot contain a chorded cycle.

  First assume that $C$ is singly chorded, so we can label $V(C) = \{c_1, c_2, c_3, c_4\}$ such that $c_1c_2c_3c_4$ is a cycle and $c_2c_4$ is the chord.
  By Lemma~\ref{RCedges}, $\|u, C\| = 3$ for every $u \in F$, and $\|c_i, F\| = 4$, for $i \in\{1, 3\}$.  Recall that $p_1,p_2 \in F$ so that $C - c_1 + p_1$ and $P - p_1 + c_1$ both contain chorded cycles, a contradiction.

So for the remainder of the proof, we assume $G[C] \cong K_4$, with $V(C) = \{c_1,c_2,c_3,c_4\}$.  Fix $u \in F$, and by Lemma \ref{struct1}, let $Q$ be a path in $R - u$ such that $F - u \subseteq V(Q)$.  Suppose $\|u,C\| = 3$, so
$\|F - u, C\| = 9$, and there exists $c \in C$ such that $c$ is adjacent to all three vertices in $F - u$.  This implies $Q + c$ and $C - c + u$ both contain chorded cycles, a contradiction. 

 Now suppose $\|u,C\| =2$ and $N_C(u) = \{c_1,c_2\}$.  Then $\|F - u, C\| = 10$, and there exist two vertices in $C$ adjacent to all three vertices in $F - u$.  If $c'$ is one of these two vertices and $c' \notin \{c_1,c_2\}$, then $Q + c'$ and $C - c' + u$ both contain chorded cycles, a contradiction.  Therefore, every vertex in $F$ is adjacent to both $c_1$ and $c_2$.  Since $\|F,C\| = 12$ and $\|u,C\| = 2$, there exists $v \in F - u$ such that $\|v,C\| = 4$.  By Lemma \ref{struct1}, there exists a path $Q'$ in $R - v$ such that $F - v \subseteq V(Q')$, so that $C - c_1 + v$ and $Q' + c_1$ both contain chorded cycles, a contradiction.

  So $\|u, C\| \in \{0, 1, 4\}$, for every $u \in F$.  Since $\|F,C\| = 12$, 
  there exists $u' \in F$ such that $\|u', C\| = 0$
  and $\|u, C\| = 4$ for every $u \in F - u'$.  By Lemma \ref{struct1}, $p_1,p_2 \in F$, so we may assume $\|p_1, C\| = 4$.  Thus, for all $c \in C$, $C - c + p_1$ is a chorded cycle, and further $\|c, P - p_1\| \le 2$, else $P - p_1 + c$ contains a chorded cycle.
  Therefore, if $\|R \setminus F, C\| > 0$,
  we can pick $c$ such that $\|c, P - p_1\| \ge 3$ so that
  $P - p_1 + c$ has a chorded cycle, a contradiction.

  Thus $\|R \setminus F, C\| = 0$.  By Lemma \ref{R6toC}, $|R| \le 5$, as otherwise we can find $F^+\subseteq V(R)$ with $|F^+| = 6$ so that for distinct $v, v' \in F^+\setminus F$, $\|\{v,v'\}, C\| \ge 1$, a contradiction.
  If $|R| = 4$, then $u'$ has a neighbor $v \in F - u'$.
  Since $R$ is not a path, by Lemma~\ref{pchords} $R \cong C_4$, so
  replacing $C$ with $C' := C - c + v$ in $\Cc$ gives a collection of $k-1$ chorded cycles
  that satisfies (O1) - (O3), but $R' := R - v + c$ has a path $P'$ such that $|P'| = |R'|$ and such that
  $u'$ is an endpoint and such that $\|u', R'\| = 1$.  
  This is a contradiction to Lemma~\ref{pchords}.

  So assume $|R| = 5$ so that $P = p_1q_1rq_2p_2$.  By Lemma~\ref{pchords}, either $p_1r,p_2r \in E(G)$, or  $R \in \{C_5, K_{2,3}\}$.   In each of these cases, we can assume that $F = \{p_1, q_1, q_2, p_2\}$, by the proof Lemma~\ref{struct1}.  Recall that $\|p_1,C\| = 4$ and $\|u', C\| = 0$ for some $u' \in F$.  Furthermore, since $\|R\setminus F, C\| = 0$, $\|r,C\| = 0$.

Suppose $R \in \{C_5, K_{2,3}\}$.  
 Let $F' := \{q_1, r, q_{2}, p_{2}\}$, so that $u' \in F'$,  $\|F', C\| \le 8$ and $\|F', R\| \le 10$.  
  Since $q_1q_2$, $rp_{2} \notin E(G)$,  $\|F', \Cc - C\| \ge 12(k - 2) + 2$ so that $k \ge 3$ and $\|F', C'\| \ge 13$ for some $C' \in \Cc- C$, a contradiction to Lemma \ref{C4bounds}.

  Thus $p_1r, p_2r \in E(G)$.
  Since three of the five vertices in $R$ send four edges to $C$,
  there exists $i \in [2]$, such that at least two vertices in $\{r,q_i,p_i\}$ have four neighbors in $C$, and so have a common neighbor $c \in C$.  This implies that $G[\{r,q_i,p_i,c\}]$ contains a chorded cycle.
  Furthermore, there exists $v \in \{p_{3-i}, q_{3-i}\}$ such that $v$ has four neighbors in $C$, and so $C - c + v$ contains a chorded cycle, a contradiction.

  Thus, $|C| \neq 4$ and  $G[C] \cong K_{3,3}$, as desired.
\end{proof}

\begin{lemma}\label{structure}
If $R \not\cong K_{2,3}$, then $G[C] \cong K_{3,3}$ for all $C \in \Cc$.  If $R \cong K_{2,3}$, then $G[C] \cong K_{3,3}$ for all but at most one $C \in \Cc$, and for any such $C$, $G[C] \cong K_{1,1,2}$ and $G[V(R) \cup V(C)] \cong K_{1,4,4}$.
%  If $R$ is not isomorphic to $K_{2,3}$ and $C \in \Cc$, then $C \cong K_{3,3}$.
%  If $R \cong K_{2,3}$, then there exists at most one $C \in \Cc$ such that $C$ is not isomorphic to $K_{3,3}$ and
%  for any such $C$, $G[C] \cong K_{2,1,1}$ and $G[V(R) \cup V(C)] \cong K_{4,4,1}$.
\end{lemma}

\begin{proof}
Let $F \subseteq V(R)$ be as guaranteed by Lemma \ref{struct1}.  If $R$ is not isomorphic to $K_{2,3}$, then $\|F,\Cc\| \ge 12(k - 1)$.  By Lemma \ref{C4bounds}, $\|F, C\| \le 12$ for all $C \in \Cc$ so that in fact, equality holds for all $C \in \Cc$.  Thus, by Lemma \ref{struct2}, $G[C] \cong K_{3,3}$ for all $C \in \Cc$.

So assume $R \cong K_{2,3}$ with partite sets $A = \{p_1,p_2,p_3\}$ and $B = \{q_1,q_2\}$ with $|A| = 3$ and $|B| = 2$.
  Since $A$ and $B$ are independent, we have $\|B, \Cc\| \ge 6k - 8$ and 
  \begin{equation*}
    2\|A, \Cc\| = \sum_{a \in A} 2\|a, \Cc\| \ge 3(6k - 2) - 12= 18k - 18,
  \end{equation*}
so $\|A, \Cc\| \ge 9(k - 1)$
  and $\|R, \Cc\| \ge 15k - 17 = 15(k - 1) - 2$.  If $\|R, C\| \ge 16$ for some $C \in \Cc$, then there exists some $u \in R$ such that $\|u,C\| = 4$.  By Lemma \ref{C4bounds}, $\|R - u, C\| \le 12$ so that there exists $u' \in R - u$ such that $\|u', C \| \le 3$.  However, $\|R - u', C \|\ge 13$, a contradiction to Lemma \ref{C4bounds}.

We therefore have that, for ever $C \in \Cc$,
 $13 \le \|R, C\| \le 15$.  Fix $C \in \Cc$.  At least two vertices in $R$ have three neighbors each in $C$ so that by Lemmas \ref{RCedges} and \ref{noC5}, $|C| =4$ or $G[C] \cong K_{3,3}$.  We claim that $G[C] \not\cong K_4$.

Suppose on the contrary, $G[C] \cong K_4$.  If $\|p_i, C\| \ge 3$ for some $i \in [3]$, Lemma \ref{C4bounds} implies that $\|R,C\| \le 12$, a contradiction.  So $\|p_i,C\| \le 2$ for all $i \in [3]$.  Hence $\|B,C\| \ge 7$ so that for all $c \in C$ and $j \in [2]$, $C - c + q_j$ is a chorded cycle.  As $\|R,C\| \ge 13$, there exists $c \in C$ such that $\|c,R\| \ge 4$.  Without loss of generality, $N_R(c) \supseteq \{p_1,p_2,q_1\}$.  However, $C - c + q_2$ and $p_1cp_2q_1p_1$ each contain chorded cycles, a contradiction.

%Thus, if $|C| = 4$, Lemma \ref{C4bounds} implies that $C$ is not isomorphic to $K_4$, which implies that $C$ is singly chorded.  So by Lemma \ref{RCedges}, $\|u, C\| \le 3$ for every $u \in R$ and $C \in \Cc$.  Since $\|A, \Cc\| = 9(k-1)$, $\|u, C\| = 3$ for every $u \in A$ and $C \in \Cc$.

So for all $C \in \Cc$, either $|C| = 4$ and $C$ is singly chorded or $G[C] \cong K_{3,3}$.  By Lemma \ref{RCedges}, $\|u, C\| \le 3$ for all $u \in A$ and $C \in \Cc$.  Since $\|A, \Cc\| \ge 9(k-1)$, we deduce that $\|A, C\| = 9$ and so $\|u, C\| = 3$ for all $u \in A$ and $C \in \Cc$.

Suppose $|C| = 4$ and $C$  is singly chorded.  We can label $V(C) = \{c_1, c_2, c_3, c_4\}$ such that $c_1c_2c_3c_4$ is a cycle and $c_2c_4$ is the chord.  By Lemma \ref{RCedges}, $uc_1,uc_3 \in E(G)$ for all $u \in A$.   Since,  $C - c_i + u$ is a chorded cycle for $i \in \{1, 3\}$, $R - u + c_i$ cannot contain a chorded cycle, which implies that $N_R(c_i) = A$.  Hence, for every $v \in B$, $N_C(v) \subseteq \{c_2, c_4\}$, and since $\|R, C\| \ge 13$, equality holds and $N_C(v)  = \{c_2, c_4\}$ for every $v \in B$.

Fix $u \in A$.  Without loss of generality, assume $N_C(u) = \{c_1,c_3,c_4\}$.  Then $C - c_2 + u$ is a chorded cycle.  If $u' \in A - u$ has $c_2 \in N_C(u)$, then $R - u + c_2$ contains a chorded cycle, a contradiction.  Thus, for all $w \in A$, $N_C(w) = \{c_1,c_3,c_4\}$  so that $N_R(c_4) = V(R)$ and $G[R \cup C] \cong K_{4,4,1}$.

Recall that $\|R, \Cc\| \ge 15(k - 1) -2$  and $\|R, C'\| \le 15$ for all $C' \in \Cc$.  Further, $\|u, C'\| \le 3$ for all $u \in R$ and $C' \in \Cc$.  Since $\|R, C\| = 13$, $\|R, C''\| = 15$ for every $C'' \in \Cc - C$.  However, for any $u \in A$, $\|u, C'\| \le 3$ so that $F:= R - u$ satisfies $\|F, C''\| \ge 12$.  Furthermore, $F$ satisfies all the hypotheses of Lemmas \ref{struct1}  and \ref{struct2}, so that $G[C''] \cong K_{3,3}$ for all $C'' \in \Cc - C$.

This completes the proof of the lemma.
\end{proof}

\begin{lemma}\label{pqchords}
  For every $u \in R$ and $C \in \Cc$, $\|u, C\| \le 3$.
  If $P'$ is path that spans $R$, $p$ is an endpoint of $P'$ and $q$
  is adjacent to $p$ on $P'$, then
  $d_G(p) = 3k-1$ and $d_G(q) \ge 3k - 1$.
  In particular, for every $C \in \Cc$ $\|p, C\| = 3$ and $\|q, C\| \ge 2$.
\end{lemma}

\begin{proof}
  Let $p$ and $p'$ be the two endpoints of $P'$, and let $q$ and $q'$ be the neighbors of $p$ and $p'$, respectively, on $P'$.  By Lemmas~\ref{RCedges} and \ref{structure}, $\|u, C\| \le 3$ for all $u \in R$ and $C \in \Cc$.  Therefore, if $d_R(u) = 2$, then $d_G(u) \le 3k - 1$, so in particular, $d_G(p) \le 3k - 1$ and $d_G(p') \le 3k - 1$.
  If $pp' \notin E$, then $d_G(p') = d_G(p) = 3k - 1$.
  Otherwise, $pp' \in E$ and $p$ is not adjacent to $q'$.  In this case, $d_R(q') = 2$ so that $d_G(p) = 3k - 1$.  Since $\|u, C\| \le 3$ for all $u \in R$ and $C \in \Cc$, it follows that $\|p, C\| = 3$.  By symmetry, this holds for $p'$ as well.

Since $\|q, R\| \le 3$, if we can show that $d_G(q) \ge 3k - 1$, it follows that $\|q, C\| \ge 2$ for all $C \in \Cc$.  So assume $d_G(q) \le 3k - 2$.   Now, $qp' \in E(G)$, as otherwise $d_G(q) \ge 3k - 1$.  If $|R| = 4$, then by Lemma \ref{pchords}, $R$ contains a chorded cycle.  So $|R| > 4$, and as a result $qq' \notin E(G)$.  Since $d_G(q) \le 3k - 2$, we get $d_G(q') \ge 3k$, and furthermore, since $d_R(q') \le 3$ and $\|q',C\| \le 3$ for all $C \in \Cc$, we deduce that $\|q',C\| = 3$ and $d_R(q') = 3$.  This implies $pq' \in E(G)$, as otherwise we get a chorded cycle in $R$.  Furthermore, $d_G(q) = 3k - 2$ and $\|q,R\| \le 3$ so that $\|q,C\| \ge 1$ for all $C \in \Cc$.

Since $|R| \ge 5$, there exists $r' \notin \{p,p'\}$ a neighbor of $q'$ on $P'$.  Note that $r' \in \Pp$ so that by the above, $d_G(r') = 3k - 1$ and $\|r', C\| = 3$ for all $C \in \Cc$.  If $|R| \ge 6$, then  $r'q \notin E(G)$ and $d_G(q) \ge 3k - 1$, a contradiction.  Hence, $|R| = 5$, and, furthermore, $R \cong K_{2,3}$
with partite sets $\{q, q'\}$ and $\{p, p', r'\}$.  Observe that for all $u \in \{p,r',q',p'\}$ and $C \in \Cc$,  $\|u, C\| = 3$.

If know fix $C \in \Cc$, such that $\|q, C\| \le 2$, which must
exist because $d(q) = 3k - 2$ and $d_R(q) = 3$.
By Lemma \ref{structure}, $G[C] \in \{K_{3,3}, K_{1,1,2}\}$.  Furthermore, if $G[C] \cong K_{1,1,2}$, then $G[C \cup R] = K_{1,4,4}$, but this contradicts
the fact that $\|q', C \cup R\| = 6$.
%, let $V(C) = \{c_1,c_2,c_3,c_4\}$ with chord $c_2c_4$.  By Lemma \ref{RCedges}, $G[\{c_3,r',p',q'\}]$ and $C - c_3 + p$ contain chorded cycles, a contradiction.
Hence, $C \cong K_{3,3}$ and let $A $ and $B$ denote its partite sets.  By Lemmas \ref{RCedges} and \ref{uvnhd}, we may assume $N_C(p) = N_C(r') = N_C(p') = A$, $N_C(q') = B$, and $N_C(q) \subseteq B$.  Since $\|q, C\| \le 2$, there exists $b \in B\setminus N_C(q)$.  We can replace $C$ with $C - b + p'$ and replace $P'$ with $bq'P'p$.  Our new collection and path satisfy (O1)-(O3).  However, $b$ is an endpoint of our new path and by the above, $d_G(b) = 3k - 1$.  Since $bq \notin E(G)$, $d_G(q) \ge 3k - 1$, a contradiction.
\end{proof}

\begin{lemma}\label{K32_K22}
  $R$ is either isomorphic to $K_{2,3}$ or $K_{2,2}$.
\end{lemma}
\begin{proof}
  If $|R| = 4$, then Lemmas~\ref{pchords} implies that $R \cong K_{2,2}$,
  so assume $|R| \ge 5$ and $R$ is not isomorphic to $K_{2,3}$.
  Let $P = u_1, \dotsc, u_{|R|}$, $p := u_1$, $q := u_2$,
  $q' := u_{|R| - 1}$ and $p' := u_{|R|}$.
  Let $C \in \Cc$.  By Lemma~\ref{structure}, $G[C] \cong K_{3,3}$, so we let $A = \{a_1,a_2,a_3\}$ and $B = \{b_1,b_2,b_3\}$ be its partite sets.
  Recall that by Lemma~\ref{RCedges}, if $\|u,C\| = 3$ for any $u \in R$, 
then $N_C(u) \in \{A, B\}$.

  First assume that $R$ is Hamiltonian (that is, $R$ contains a cycle of size $|R|$).  Since every vertex in $R$ is the endpoint of a path spanning $R$, 
  by Lemma~\ref{pqchords}, $\|u, C\| = 3$ for every $C \in \Cc$ and $u \in R$.
  By Lemma~\ref{uvnhd}, we can assume that
  $N_C(u_i)  = A$ if $i$ is odd and $N(u_{i}) = B$ is $i$ is even.
  Therefore, Lemma~\ref{uvnhd} implies that $|R|$ is even, which further implies that $|R| \ge 6$.
  Then for any $a \in A$ and $b \in B$, 
  $G[\{u_1, \dotsc, u_4, a, b\}]$ and 
  $C - a - b + u_5 + u_6$ contain chorded cycles, a contradiction. 

  So we can assume $R$ is not Hamiltonian.  
  Let $pw$ be a hop on $P$ so that  $w \neq p'$.
  First assume $w \neq q'$.
  Without loss generality assume that $N_C(p') = A$. 
  By Lemmas~\ref{uvnhd} and \ref{pqchords}, $N_C(p) \cap N_C(q) = \emptyset$, and so there exists $cc' \in E(C)$ such that   $pcc'qPwp$ is a cycle with chord $pq$.
  By Lemmas~\ref{uvnhd} and \ref{pqchords},
  $|N_C(p') - c - c'| \ge 2$ and 
  $|N_C(q') - c - c'| \ge 1$, so $C - c - c' + p' + q'$ contains a chorded cycle, a contradiction.

  Now we can assume that both $pq'$ and $qp'$ are edges.
  Since $R \neq K_{2,3}$, we have that $|R| \ge 6$.  Let  $r \neq p$ and $r' \neq p'$  be the neighbors of $q$ and $q'$, respectively, on $P$.  Note that $r$ and $r'$ are endpoints of paths spanning $R$ so that $\|r,C\| = \|r',C\| = 3$.  By Lemmas \ref{uvnhd} and \ref{pqchords}, and because $pq',qp' \in E(G)$, we may assume that $N_C(p) = N_C(r) = N_C(r') = N_C(p') = A$ and $N_C(q) \cup N_C(q') \subseteq B$.  In particular, we may assume $qb_1 \in E(G)$ so that $pa_1b_2a_2b_1qp$ is a cycle with chord $pa_2$, and $rPp'a_3r$ is a cycle with chord $a_3r'$, a contradiction.

So $|R| = 5$ and $R \cong K_{2,3}$, as desired.
\end{proof}

\begin{lemma}\label{6only}
  If  $G[C] \cong K_{3,3}$ for every $C \in \mathcal{C}$, then $G \cong G_1(n,k)$. 
\end{lemma}
\begin{proof}
By Lemma \ref{K32_K22}, $R \in \{K_{2,2}, K_{2,3}\}$.  So let $U_1,U_2 \subseteq V(R)$ be the partite sets of $R$ such that $|U_1| \ge |U_2| = 2$, and let $u_1 \in U_1$, $V_2 := N_G(u_1)$, and $V_1 := V(G) \setminus V_2$.  Since  $u_1$ is the end of spanning path of $R$, Lemma~\ref{pqchords} implies that $|V_2| = 3k - 1$.  Since $|G| \le 6(k - 1) + 5$, $|V_1| \le 3k$.  We aim to show that $N_G(v) = V_2$ for all $v \in V_1$.  This will imply that $G \cong G_1(n,k)$.

Fix $v \in V_1 - u_1$. Since $u_1v \notin E(G)$, Lemma~\ref{pqchords} implies that $d_G(v) \ge 3k - 1$.  If $v \in U_1$, then $v$ is the end of a spanning path of $R$, and by Lemmas \ref{RCedges}, \ref{uvnhd} and \ref{pqchords},
  $N_G(v) = N_G(u_1) = V_2$.  So we may assume $v \in V_1\setminus U_1$, and in particular, $v \in C$ for some $C \in \Cc$.

  Define $V_1' := \{u \in V_1 : \|u, U_2\| \ge 1\}$,  and suppose $v \in V_1' \setminus U_1$.  Recall that we are assuming $G[C] \cong K_{3,3}$ for all $C \in \Cc$ so that by Lemma~\ref{RCedges}, $G[C - v + u_1] \cong K_{3,3}$.  Furthermore, 
  $v$ is an end of a path of length $|R|$ in $R' := R - u_1 + v$.  This new collection and path satisfy (O1)-(O3), so by Lemma~\ref{K32_K22}, $R' \cong R$ and $N_G(v) = N_G(u_1) =  V_2$.

  Now suppose $v \in V_1 \setminus V_1'$.
  Since $d_G(v) \ge 3k - 1$ and $v$ has at most $3(k - 1)$ neighbors in $V_2$, $v$ must have two neighbors in $V_1$.
  By Lemmas~\ref{uvnhd} and \ref{pqchords}, for every $u_2 \in U_2$,
  $d_G(u_2) \ge 3k - 1$ and $N_G(u_2) \subseteq V_1$, so that
  $|V_1'| \ge 3k - 1$.  Since $|V_1| \le 3k$, $v$ has a neighbor, say $v'$, in $V_1'$.  However, by the above, $N_G(v') = V_2$, which contradicts the fact that $vv'$ is an edge.
  Therefore, $V_1' = V_1$ which finishes the proof of the lemma. 
\end{proof}

\begin{lemma}\label{one4}
Suppose there exists $C \in \mathcal C$ with $|C|=4$. Then $G \cong G_2(k)$.
\end{lemma}
\begin{proof}
  By Lemmas~\ref{structure} and \ref{K32_K22}, we can assume $R \cong K_{2,3}$, $G[C] \cong K_{1,1,2}$, and $G[R \cup C] \cong K_{1,4,4}$. Let $A'$ and $B'$ be the two partite sets of size four and $\{c\}$ be the partite set of size one in $G[R \cup C]$.  
  By symmetry, we can assume that any $v \in A' \cup B'$ is an 
  end of a spanning path in $R$ or the end of a spanning path
  of $G[V(G) \setminus V(\Cc')]$ for some collection $\Cc'$
  of $k-1$ vertex disjoint cycles that satisfies (O1)-(O3), 
  so, by Lemma~\ref{pqchords},
  $d_G(v) = 3k - 1$ and $\|v, \Cc - C\| = 3(k-2)$.
  By Lemma \ref{structure}, for all $D \in \Cc - C$,  
  $G[D] \cong K_{3,3}$, and, with Lemma \ref{uvnhd}, 
  we deduce that $\|v, D\| = 3$ and that
  we can label the partite sets of $D$ as $A_D$ and $B_D$
  so that for every $p \in A'$, $N_D(p) = B_D$ and
  for every $q \in B'$, $N_D(q) = A_D$.
  Therefore, there exists a partition $\{A, B, \{c\}\}$ of $V(G)$ such
  that 
  for every $p \in A'$, $N_G(p) = B + c$, 
  for every $q \in B'$, $N_G(q) = A + c$, and
  $|A| = |B| = 3k - 2$.

  If $u \in V(G) \setminus (A' \cup B')$, then there exists $D \in \Cc - C$,
  such that $u \in D$. 
  Let $p \in A' \cap V(R)$, 
  and $q \in B' \cap V(R)$
  and label $\{w, w'\} = \{p, q\}$ 
  so that $uw \notin E(G)$ and $uw' \in E(G)$.
  We have that
  $G[D - u + w] \cong K_{3,3}$
  and $G[R - w + u] \cong K_{3,2}$,
  so there exists a collection $\Cc'$ 
  of $k-1$ vertex disjoint cycles containing $C$ that satisfies (O1)-(O3),
  and there exists a spanning path of 
  of $G[V(G) \setminus V(\Cc')]$ such that
  $u$ is an endpoint or $u$ is the neighbor of an endpoint.
  Therefore, by Lemma~\ref{pqchords}, $d_G(u) \ge 3k - 1$, 
  so, with Lemma~\ref{structure}, we have that
  $N_C(u) = (V(C) \setminus N_C(w')) + c$ 
  and, for any $D' \in \Cc' - C$, 
  by Lemma~\ref{uvnhd}, 
  $N_{D'}(u) = D' \setminus N_{D'}(w')$.
  Therefore, 
  either $N_G(u) \supseteq B + c$ if $u \in A$ or
  $N_G(u) \supseteq A + c$ if $u \in B$.  Hence, 
% By Lemmas~\ref{structure} and \ref{K32_K22}, we can assume $R \cong K_{2,3}$, $G[C] \cong K_{1,1,2}$, and $G[R \cup C] \cong K_{1,4,4}$. Let $A'$ and $B'$ be the two partite sets of size four in $R$, and let $\{c\}$ be the partite set of size one in $G[R \cup C]$.  
%Fix $D \in \Cc - C$ and let $d \in D$.  There always exists $w \in \{p_1,q_1\}\setminus N_G(d)$.  Then $G[D-d+w] \cong K_{3,3}$  and $G[R - w + d] \cong K_{2,3}$. 
%  So, by Lemma~\ref{structure}, if $w \in A$, then $N_G(w)=B \cup \{c\}$, and if $w \in B$, then $N_G(w)=A \cup \{c\}$.  Thus, $G[A \cup B] \cong K_{|A|,|B|}$.  Now, $|B| = d_G(p_1) - 1 = 3k - 2$, and $|A| = d_G(q_1) - 1 = \|q_1,\Cc\| + 3 - 1 = 3(k - 2) + 2 + 2 = 3k - 2$.  Since $c$ dominates $G$, 
  $G$ contains $G_2(k)$ as a spanning subgraph.  As $G_2(k)$ is edge-maximal with respect to not containing $k$ disjoint chorded cycles, $G \cong G_2(k)$.
\end{proof}

Using Lemmas~\ref{structure}, \ref{K32_K22}, \ref{6only}, and \ref{one4}, we conclude $G \in \{G_1(n,k),G_2(k)\}$.

%%%%%%%%%%%%%%%%%%%%%%%%%%%%%%%%%%
%%%%%%%%%%%%%%%%%%%%%%%%%%%%%%%%%%
%%%%%%%%%%%%%%%%%%%%%%%%%%%%%%%%%%
\section{Concluding Remarks}\label{remarks}
%%%%%%%%%%%%%%%%%%%%%%%%%%%%%%%%%%
%%%%%%%%%%%%%%%%%%%%%%%%%%%%%%%%%%
%%%%%%%%%%%%%%%%%%%%%%%%%%%%%%%%%%

Many variations on Theorems~\ref{CHT} and \ref{FinkelT}  have appeared, and suggest further extensions of Theorem~\ref{main}. We present only a small selection below.

A result of Gould, Hirohata, and Horn \cite{GHH} implies the following:
\begin{thm}
Let $G$ be a graph on $|G| \geq 6k$ vertices with $\delta(G) \geq 3k$. Then $G$ contains $k$ disjoint doubly chorded cycles.
\end{thm}

While it is not clear that $|G| \ge 6k$ is necessary, it would be interesting to characterize the sharpness examples for this theorem; that is, if 
$|G| \geq 6k$ and $\delta(G)=3k-1$ but $G$ does not contain $k$ disjoint doubly chorded cycles, what does $G$ look like?  For more results on the existence of $k$ disjoint multiply chorded cycles, see \cite{GHM}

Additionally, rather than consider $\delta(G)$ or $\sigma_2(G)$, one may consider the neighborhood union, $\min\{|N(x) \cup N(y)| : xy \in E(\overline{G})\}$.  See the following results.

\begin{thm}[Faudree-Gould, \cite {FG}]
If $G$ has $n \geq 3k$ vertices and $|N(x) \cup N(y)| \geq 3k$ for all nonadjacent pairs of vertices $x,y$, then $G$ contains $k$ disjoint cycles.
\end{thm}

\begin{thm}[Gould-Hirohata-Horn, \cite{GHH}]
Let $G$ be a graph on at least $4k$ vertices such that for any nonadjacent $x,y \in V(G)$, $|N(x) \cup N(y)| \geq 4k+1$. Then $G$ contains $k$ disjoint chorded cycles.
\end{thm}

\begin{thm}[Gould-Hirohata-Horn, \cite{GHH}]
Let $G$ be a graph on {$n>30k$} vertices such that for any nonadjacent $x,y \in V(G)$, $|N(x) \cup N(y)| \geq {2k+1}$. Then $G$ contains $k$ disjoint cycles.
\end{thm}

\begin{thm}[Qiao, \cite{Q}]
Let $r,s$ be nonnegative integers, and let $G$ be a graph on at least $3r+4s$ vertices such that for any nonadjacent $x,y \in V(G)$, $|N(x) \cup N(y)| \geq 3r+4s+1$. Then $G$ contains $r+s$ disjoint cycles,  $s$ of them chorded.
\end{thm}

\section{Acknowledgements}
The authors would like to thank Ron Gould, Megan Cream, and Michael Pelsmajer for productive conversations about this problem.


\begin{thebibliography}{10}

\bibitem{BFG}
A. Bialostocki, D. Finkel, and A. Gy\'arf\'as, 
Disjoint chorded cycles in graphs,  
{\em Discrete Mathematics} 308 (2008), no. 23, 5886-5890. 

%\bibitem{CLW}B.-L. Chen, K.-W. Lih, and P.-L. Wu, Equitable coloring
%and the maximum degree, \emph{Europ. J. Combinatorics}, 15 (1994)
%443--447.

\bibitem{CFGL}
S. Chiba, S. Fujita, Y. Gao, and G.  Li,
On a sharp degree sum condition for disjoint chorded cycles in graphs.
{\em Graphs Combin.} 26 (2010), no. 2, 173-186. 

\bibitem{CH} 
K. Corr\' adi and A. Hajnal, On the maximal number of independent circuits in a graph, 
{\em Acta Math. Acad. Sci. Hungar.} 14 (1963) 423--439.

\bibitem{Di} 
G. Dirac, 
Some results concerning the structure of graphs, 
{\em Canad. Math. Bull.}  6 (1963) 183--210.

%\bibitem{DE}G. Dirac and P. Erd\H os, On the maximal number of independent circuits
%in a graph,
%{\em Acta Math. Acad. Sci. Hungar.} 14 (1963) 79--94.

\bibitem{E} 
H. Enomoto, 
On the existence of disjoint cycles in a graph, 
{\em Combinatorica} 18(4) (1998) 487-492.

\bibitem{FG}
J. Faudree and R. Gould,  
A note on neighborhood unions and independent cycles, 
{\em Ars Combin.} 76 (2005), 29-31. 05C69

\bibitem{F} 
D. Finkel,
On the number of independent chorded cycles in a graph,
{\em Discrete Mathematics} 308 (2008) no 22, 5265-5268.

\bibitem{GHH}
R. Gould, K. Hirohata, and P.  Horn,
Independent cycles and chorded cycles in graphs,
{\em J. Comb.} 4 (2013), no. 1, 105–122. 

\bibitem{GHM}
R. Gould, P. Horn, and C. Magnant, 
Multiply chorded cycles,
{\em SIAM J. Discrete Math.} 28 (2014), no. 1, 160-172.


%\bibitem{KK2}H. A. Kierstead and A. V. Kostochka, A refinement of
%a result of Corr\' adi and Hajnal,  {\em J. Graph Theory}, to appear.

\bibitem{KKY}
H. A. Kierstead, A. V. Kostochka, and E. C. Yeager,
On the Corr\' adi-Hajnal Theorem and a question of Dirac,
{\em Journal of Combinatorial Theory, Series B}, to appear.

\bibitem{KKY2}
H. A. Kierstead, A. V. Kostochka, and E. C. Yeager,
The $(2k-1)$-connected multigraphs with at most $k-1$ disjoint cycles,
{\em Combinatorica}, to appear.

\bibitem{KKMY}
H. A. Kierstead, A. V. Kostochka, T.N. Molla, and E. C. Yeager,
Sharpening an Ore-type version of the Corr\'adi-Hajnal Theorem,
{\em submitted}.

%\bibitem{Lo} L. Lov\' asz, On graphs not containing independent circuits,
%(Hungarian. English summary)
%Mat. Lapok 16 (1965), 289--299.

\bibitem{Q}
S. Qiao,  
Neighborhood unions and disjoint chorded cycles in graphs, 
{\em Discrete Mathematics} 312 (2012), no. 5, 891-897.

\bibitem{W} 
H. Wang, 
 On the maximum number of disjoint cycles in a graph, 
{\em Discrete Mathematics} 205 (1999) 183-190.

\end{thebibliography}
\end{document}